\documentclass[12pt,oneside,english]{amsart}
\usepackage[T1]{fontenc}
\usepackage[latin9]{inputenc}
\usepackage{geometry}
\usepackage{amsthm}
\usepackage{amstext}
\usepackage{amssymb}
\usepackage{graphicx}
\usepackage{setspace}
\PassOptionsToPackage{normalem}{ulem}
\usepackage{ulem}
\makeatletter
\numberwithin{equation}{section}
\numberwithin{figure}{section}
\theoremstyle{plain}
\newtheorem{thm}{\protect\theoremname}[section]
  \theoremstyle{definition}
  \newtheorem{example}[thm]{\protect\examplename}
  \theoremstyle{plain}
  \newtheorem{cor}[thm]{\protect\corollaryname}
  \theoremstyle{remark}
  \newtheorem{rem}[thm]{\protect\remarkname}
  \theoremstyle{plain}
  \newtheorem{lem}[thm]{\protect\lemmaname}

\usepackage{lscape}
\usepackage{lscape}
\usepackage{lineno}

\makeatother

\usepackage{babel}
  \providecommand{\corollaryname}{Corollary}
  \providecommand{\examplename}{Example}
  \providecommand{\lemmaname}{Lemma}
  \providecommand{\remarkname}{Remark}
\providecommand{\theoremname}{Theorem}

\begin{document}

\title{Short Title: Chicken Walks}

\maketitle
\textbf{\Large Understanding Chicken Walks on $n\times n$ Grid: Hamiltonian
Paths, Discrete Dynamics and Rectifiable Paths}{\Large \par}

\vspace{1cm}

\textbf{Appeared in}\textbf{\emph{ Mathematical Methods in the Applied
Sciences (Wiley-Blackwell) DOI: 10.1002/mma.3301 }}

\vspace{1cm}

\begin{center}

\author{ARNI S.R. SRINIVASA RAO{*}}

\vspace{0.2cm}

Georgia Regents University,

1120 15th Street, Augusta, GA 30912, USA

Email address: arrao@gru.edu

$ $

and

$ $

Bayesian and Interdisciplinary Research Unit, 

Indian Statistical Institute, Kolkata 700108

\end{center}

$ $\vspace{0.1cm}

\begin{center}

\author{FIONA TOMLEY and DAMER BLAKE }

\vspace{0.2cm}

The Royal Veterinary College,

University of London, Hatfield Herts AL9 7TA, UK

\end{center}

\tableofcontents{}

\vspace{0.1cm}

{*}Corresponding author.
\begin{abstract}
Understanding animal movements and modelling the routes they travel
can be essential in studies of pathogen transmission dynamics. Pathogen
biology is also of crucial importance, defining the manner in which
infectious agents are transmitted. In this article we investigate
animal movement with relevance to pathogen transmission by physical
rather than airborne contact, using the domestic chicken and its protozoan
parasite \emph{Eimeria} as an example. We have obtained a configuration
for the maximum possible distance that a chicken can walk through
straight and non-overlapping paths (defined in this paper) on square
grid graphs. We have obtained preliminary results for such walks which
can be practically adopted and tested as a foundation to improve understanding
of non-airborne pathogen transmission. Linking individual non-overlapping
walks within a grid-delineated area can be used to support modeling
of the frequently repetitive, overlapping walks characteristic of
the domestic chicken, providing a framework to model faecal deposition
and subsequent parasite dissemination by faecal/host contact.We also
pose an open problem on multiple walks on finite grid graphs. These
results grew from biological insights and have potential applications.
\end{abstract}

\section{Straight Walk and Non-overlapping Walk }

Parasitic pathogens with direct single-host life cycles rarely rely
on aerial transmission for dissemination, more commonly featuring
direct (i.e. physical contact) or indirect (environmental, food- or
water-borne) routes \cite{Taylor2007}. Examples include protozoans
such as \emph{Cryptosporidium} and \emph{Eimeria}, helminths such
as \emph{Ostertagia ostertagi} and arthropods such as \emph{Sarcoptes
scabei}. Understanding the transmission of such pathogens requires
an awareness of host movement as the initial source of pathogen spread,
informed by subsequent environmental factors such as food movement,
flow of water and other fomites. Recognition of the relevance of poultry
to food security has elevated the importance of their pathogens, with
parasites such as \emph{Eimeria} of key significance \cite{Chapman2013}.
\emph{Eimeria} can cause the disease coccidiosis, a severe enteritis
characterised by high morbidity and, sometimes, mortality. The global
cost of losses attributed to \emph{Eimeria} and their control has
been estimated to exceed \$3 billion per annum, complicated further
by welfare implications \cite{Dalloul2006-1}. Most \emph{Eimeria}
are absolutely host-specific and exhibit a strict faecal-oral lifecycle
including an environmental stage, called the oocyst, which must undergo
a process termed sporulation over twelve to thirty hours external
to the host in order to become infective. Thus, the physical behaviour
of chickens including the amount of time spent moving, the distance
moved, the pattern of movement and the frequency and location of defaecation
whilst moving are of critical importance to understanding \emph{Eimeria}
transmission. Transmission rates have previously been calculated for
\emph{Eimeria} \emph{acervulina} \cite{Velkers}. Overlaying these
data onto models of chicken movement will support prediction of \emph{Eimeria}
transmission through a flock, facilitating scrutiny of the impact
of management system and the opportunity for co-infection by genetically
diverse parasite strains \cite{Shirley,Williams,Willi}. The frequency
of co-infection with genetically diverse strains will determine the
rate at which cross fertilization may occur, influencing the occurrence
of novel genotypes with relevance to evasion from drug- and vaccine-mediated
parasite killing \cite{Damer}. Inspired by the importance of chicken
movement in \emph{Eimeria} transmission, this work grew into an exercise
to model chicken movement while studying the length of distance chickens
walk per unit time in a pen, their parasite disseminating characteristics
and the rate at which infection spreads between birds. In order to
understand the complexity of chicken movement we have begun by assuming
a square pen which can be subdivided into a cellular graph. The walks
considered in this manuscript are of maximum length. By joining several
such walks together in the future we will begin to recreate multiple
chicken paths as an entrée to modeling chicken movements in more complex
environments.

Let us consider an area, $S$, of dimension $n\times n$ ($n>1)$
which is divided into $n\times n-$small squares (or cells). Let $\left(i,j\right)$
be the cell which is located at $i^{th}$ row and $j^{th}$ column
of these $n$ cells. Suppose we leave a chicken in one of the cells
of $S$ and suppose we are interested in observing the walking behaviour
of chicken through the following two rules, i) Walking from one corner
point to a neighboring corner point and ii) Walking only through each
cell (excluding on the cell boundaries). The $(i,j)^{th}$ cell is
denoted by $S_{ij}\left[A_{i},B_{j},C_{j},D_{i}\right]$, where $A_{i}$,
$B_{j}$, $C_{j}$ and $D_{i}$ are four vertices of this cell which
are located at the upper left corner, upper right corner, lower right
corner and lower left corner, respectively. A chicken sitting inside
the cell $(i,j)$ (not on the vertices) is denoted by $K(i,j)$ and
a chicken sitting on the vertices $A_{i}$, $B_{j}$, $C_{j}$ and
$D_{i}$ of $S_{ij}$ is denoted by $K(A_{i})$, $K(B_{j})$, $K(C_{j})$
and $K(D_{i})$ of $S_{ij}$, respectively. A \emph{straight walk
}by $K(i,j)$ is defined here as a walk initiated by $K(i,j)$ for
all $i=1,2,\cdots,n$ and $j=1,2,\cdots,n$ by moving to neighboring
cell through adjacent sides only and a \emph{straight walk }by $K(A_{i})$
or $K(B_{j})$ or $K(C_{j})$ or $K(D_{i})$ of $S_{ij}$, respectively,
for all $i=1,2,\cdots,n$ and $j=1,2,\cdots,n$ is defined here as
a walk from one cell to another cell that shares an edge with the
current cell. For example, $K(1,3)$ means that the chicken is in
the cell which is at first row and third column and $K(A_{2})$ of
$S_{23}$ means chicken is at the vertex $A_{2}$ of cell $S_{23}$
(which is located at second row and third column) which has four vertices
$\left[A_{2},B_{3},C_{3},D_{2}\right]$. 

We can visualize the area $S$ either with even number of cells $\left(2n\times2n\right)$
or with odd number of cells $\left((2n+1)\times(2n+1)\right)$ and
is placed on a \emph{grid graph}, $G$, which is a subset of an \emph{infinite
graph, $G^{\infty}$ \cite{ITAI}.} See \cite{ITAI,Zami,Kesh,kwo,Thom}
for foundations on grid graphs and \cite{NASH,NASH2,BON,DS,HAR,RODL}
for infinite graphs. If an area $S$ has\emph{ }$\left(2n\times2n\right)$
cells then it will have $\left((2n+1)\times(2n+1)\right)$ \emph{vertices.}
This gives us some flexibility to construct walks connecting some
finite number of cells and relate such walks\emph{ }to the walks through
vertices. Using the same flexibility, we define a cell as \emph{even}
if both $i$ and $j$ are even or $i+j$ $\cong$ $0$$\mbox{(\ensuremath{\mbox{mod}}2)}$.
Hence, a maximum possible walk between two cells $(i,j)$ and $(i^{*},j^{*})$
can be considered as an \emph{Hamiltonian Path }between these two
cells. \emph{The problem of determining if a given graph G has a Hamiltonian
path is NP-Complete \cite{ITAI}}. We have described Hamiltonian and
related paths through cells in a grid in section 2. The maximum paths
between cells that we considered as described above and further discussed
in section 3 and 4 are simpler situations than NP-complete problems.
Our results indicate maximum possible walks can be configured based
on the position of the cells connecting walks in an even dimensional
area and an odd dimensional area. Primarily we differ in our approach
because we tried all our attempts by connecting maximum possible walks
between two cells. However, one can attempt to relate particular cases
of our types of walks with \emph{Hamiltonian path} configurations.

\section{Related Works}

Our results were not inspired by previous work on \emph{Hamiltonian
Paths or NP-Complete} problems. We obtained the solutions of maximum
possible walks from fundamental principles while trying to model \emph{chicken
walks }to understand transmission rates and cross fertilization of
certain parasites with strict fecal / oral life cycles among chickens.
We have thought of distributing the locations of defecations per unit
of time and hence we tried to link two \emph{Hamiltonian paths} at
these locations. Moreover, \emph{Hamiltonian Path} problems are related
to the paths connected between two vertices. See \cite{Bellman1962,Rubin1974}
for basic introduction to the \emph{Hamiltonian paths}. Let $G$ be
a finite and simple graph with at least $3$ vertices. Then, by Ore's
Theorem \cite{Ore1960}, $G$ is \emph{Hamiltonian, }if for every
pair of non-adjacent vertices (say, $a$ and $b$), the sum of the
degrees of $a$ and $b$ is at least $3.$ Ore's Theorem is based
on the arguments of the work by Newman \cite{Newman1958} who proved
that \emph{``Any graph with $2n$ vertices each of order not less
than n must contain a $2n-gon$''. }A graph $G$ is called Ore-type
$(k)$ if it satisfies $d(a)+d(b)\geq\left|G\right|+k$, where $d(a)$
and $d(b)$ are degrees of $a$ and $b,$ respectively. $G$ is $k-path$
\emph{Hamiltonian }if $G$ is a graph on $p$ vertices and $d(a)+d(b)\geq p+k$
for every pair $\left\{ a,b\right\} $ \cite{Kronk}. In general,
when $G$ has $p$ vertices, then $G$ is $k-path$ \emph{Hamiltonian}
if $G$ has at least $\frac{1}{2}(p-1)(p-2)+k+2$ edges \cite{Kronk}.
This condition is sufficient for a graph to be $k-path$ \emph{Hamiltonian}.
For works on the longest paths in undirected graphs (random) refer
to \cite{Ajtai,Pittel,Krivlevich2013}. Algorithms for approximating
the longest paths in grid graphs and meshes can be seen here \cite{Karger,Feder,Zhang2011,Fatemah,Fatemah2013}.
There are methods which are based on the longest paths in random graphs
(for example, see \cite{Ajtai}) and search for the trees formed by
probability processes \cite{Fernandez,Krive2010}. Using the Turing
machine-based models, computational complexity of $k-path$ problems
were studied (see \cite{Chou2005}) and for the importance of finding
a path in a plane, see \cite{Henrici1986}.

\section{Maximum Possible Walk }

In this section we study the properties of obtaining maximum possible
walks under the hypotheses of straight and non-overlapping walks.
\begin{thm}
\label{thm:1}(A) Suppose a straight walk is initiated by $K(i,j)$
in $S$ (the maximum distance covered by $K(i,j)$ without stepping
onto the same cell cannot exceed $n^{2}-1$), then there exists configurations
when the walk is initiated through any neighboring side of the $K(i,j)$. 

(B) Suppose a straight walk is initiated by $ $$K(A_{i})$ or $K(B_{j})$
or $K(C_{j})$ or $K(D_{i})$ of $S_{ij}$ in $S$ (the maximum distance
covered by each of these walks cannot exceed $(n+1)^{2}-1$), then
there exists configurations when the walk is initiated through any
neighboring vertex. \end{thm}
\begin{proof}
That maximum distance travelled is $n^{2}-1$\textbf{ }is easy to
verify, so we will concentrate here on configurations. \textbf{(A)
}We introduce notations for the directions for movement of a chicken
between cells either row wise or column wise. A chicken moved from
$(i,j)$ to $(i-1,j)$ is denoted by the direction $_{(i-1)j}d_{ij}$,
similarly a move from $(i,j)$ to $(i,j+1)$ is denoted by the direction
$_{i(j+1)}d_{ij}$, move from $(i,j)$ to $(i+1,j)$ is denoted by
the direction $_{(i+1)j}d_{ij}$, move from $(i,j)$ to $(i,j-1)$
is denoted by the direction $_{i(j-1)}d_{ij}$. 

We prove the theorem in two situations, (I) when $S$ has dimension
$2n\times2n$ and (II) when $S$ has dimension $(2n+1)\times(2n+1)$ 

(I) \textbf{$S$ has dimension $2n\times2n$. }Consider a chicken
in an arbitrary cell, $(i',j')$ i.e. $K(i',j')$. Suppose $(2n-i')$
is an odd number, $(2n-j')$ is an even number. This means there are
an odd number of columns to the right of $K(i',j')$, an even number
of columns to the left of $K(i',j')$ and an odd number of rows above
$K(i',j')$, an even number of rows below $K(i',j')$. We prove the
statement for each of the four directions. 

\textbf{(a)} Starting direction from $K(i',j')$ is $_{(i'-1)j'}d_{i'j'}$.
Follow the configuration given in the steps shown below:

($a_{1}$) take $(i'-1)$ steps in the direction $_{(i'-1)j'}d_{i'j'}$
to reach the first row, ($a_{2}$) take $(2n-i')$ steps in the direction
$_{1(2n-i'+1)}d_{1(2n-i')}$ to reach the last column, ($a_{3}$)
take $(2n-1)$ steps in the direction $_{2(2n)}d_{1(2n)}$ to reach
the last row, ($a_{4}$) take one step in the direction $_{(2n)(2n-1)}d_{(2n)(2n)}$,
($a_{5}$) take $(2n-2)$ steps in the direction $_{(2n-1)(2n-1)}d_{(2n)(2n-1)}$,
($a_{6}$) take one step in the direction $_{2(2n-2)}d_{2(2n-1)}$,
($a_{7}$) take $(2n-2)$ steps in the direction $_{(2n)(2n-2)}d_{2(2n-2)}$
to reach last row, ($a_{8}$) repeat the steps similar to the steps
($a_{4}$) to ($a_{6}$) to reach the last row and $(2n-j'+1)$ column
where the given chicken is currently located, i.e. $K\left(2n,(2n-j'+1)\right)$,
($a_{9}$) take $(2n-j')$ steps in the direction $_{(2n)(2n-j')}d_{(2n)(2n-j'+1)}$
$ $ to reach the first column, ($a_{10}$) take $(2n-1)$ steps in
the direction $_{(2n-1)1}d_{(2n)1}$ to reach the first row, ($a_{11}$)
take one step in the direction $_{12}d_{11}$, ($a_{12}$) take $(2n-2)$
steps in the direction $_{22}d_{12}$, ($a_{13}$) take one step in
the direction $_{(2n-1)3}d_{(2n-1)2}$, ($a_{14}$) take $(2n-2)$
steps in the direction $_{(2n-2)3}d_{(2n-1)3}$ to reach first row,
($a_{15}$) take one step in the direction $_{14}d_{13}$, ($a_{16}$)
take $(2n-1)$ steps in the direction $_{24}d_{14}$, ($a_{17}$)
repeat the steps similar to the steps ($a_{8}$) to ($a_{16}$) such
that the chicken is located in the $(2n-1)$ row and $(2n-j'-1)$
column i.e. $K\left((2n-1),(2n-j'-1)\right)$, ($a_{18}$) take one
step in the direction $_{(2n-1)(2n-j')}d_{(2n-1)(2n-j'-1)}$, ($a_{19}$)
take $(i'-2)$ steps in the direction $_{(2n-2)(2n-j')}d_{(2n-1)(2n-j')}$
to reach the cell $\left((2n-i'+1\right),(2n-j')$ such that we will
have $K\left((2n-i'+1\right),(2n-j')$. This way the chicken takes
$4n^{2}-1$ steps, and we achieved maximum distance configuration. 

\textbf{(b)} Starting direction from $K(i',j')$ is $_{(i')(j'+1)}d_{i'j'}$.
Maximum distance configuration is given in the steps shown below:

($b_{1}$) take $(2n-j')$ steps in the direction $_{(i')(j'+1)}d_{i'j'}$
to reach the last column, ($b_{2}$) take $(2n-i')$ steps in the
direction $_{(2n-i'+1)(2n)}d_{(2n-i')(2n)}$ to reach the last row,
($b_{3}$) take $(2n-1)$ steps in the direction $_{(2n)(2n-1)}d_{(2n)(2n)}$
to reach the first column, ($b_{4}$) take one step in the direction
$_{(2n-1)1}d_{(2n)1}$, ($b_{5}$) take $(2n-2)$ steps in the direction
$_{(2n-1)2}d_{(2n-1)1}$, ($b_{6}$) take one step in the direction
$_{(2n-2)(2n-1)}d_{(2n-1)(2n-1)}$, ($b_{7}$) take $(2n-2)$ steps
in the direction $_{(2n-2)(2n-2)}d_{(2n-2)(2n-1)}$ to reach the first
column, ($b_{8}$) take one step in the direction $_{(2n-3)1}d_{(2n-2)1}$,
($b_{9}$) take $(2n-2)$ steps in the direction $_{(2n-3)2}d_{(2n-3)1}$,
($b_{10}$) repeat the steps similar to the steps ($b_{6}$) to ($b_{9}$)
such that the chicken is located in the cell $\left((2n-i'+3),(2n-1)\right)$,
i.e. $K\left((2n-i'+3),(2n-1)\right)$, ($b_{11}$) take two steps
in the direction $_{(2n-i'+1)(2n-1)}d_{(2n-i'+3)(2n-1)}$, ($b_{12}$)
take one step in the direction $_{(2n-i'+1)(2n-2)}d_{(2n-i'+1)(2n-1)}$,
($b_{13}$) take one step in the direction $_{(2n-i'+2)(2n-2)}d_{(2n-i'+1)(2n-2)}$,
($b_{14}$) take one step in the direction $_{(2n-i'+2)(2n-3)}d_{(2n-i'+2)(2n-2)}$, 

($b_{15}$) take one step in the direction $_{(2n-i'+1)(2n-3)}d_{(2n-i'+2)(2n-3)}$,
($b_{16}$) repeat the steps similar to the steps ($b_{7}$) to ($b_{15}$)
to reach the cell $\left((2n-i'+1),1\right)$, ($b_{17}$) continue
for $i'$ steps in the same direction to reach the cell $(1,1)$,
($b_{18}$) take $(2n-1)$ steps in the direction $_{12}d_{11}$ to
reach the last column, ($b_{19}$) take one step in the direction
$_{2(2n)}d_{1(2n)}$, ($b_{20}$) take $(2n-2)$ steps in the direction
$_{22}d_{2(2n)}$, ($b_{21}$) take one step in the direction $_{32}d_{22}$,
($b_{22}$) take $(2n-2)$ steps in the direction $_{33}d_{32}$ to
reach the last column, ($b_{23}$) repeat the steps similar to the
steps ($b_{19}$) to ($b_{22}$) such that the chicken is located
in the cell $\left((i'-3),2n\right)$, ($b_{24}$) take two steps
in the direction $_{(i'-2)(2n)}d_{(i'-3)(2n)}$, ($b_{25}$) take
one step in the direction $_{(i'-1)(2n-1)}d_{(i'-1)(2n)}$, ($b_{26}$)
take one step in the direction $_{(i'-2)(2n-1)}d_{(i'-1)(2n-1)}$,
($b_{27}$) take one step in the direction $_{(i'-2)(2n-2)}d_{(i'-2)(2n-1)}$,
($b_{28}$) take one step in the direction $_{(i'-1)(2n-2)}d_{(i'-2)(2n-2)}$,
($b_{29}$) repeat the steps similar to the steps ($b_{25}$) to ($b_{28}$)
to reach the cell $\left((i'-1),2\right)$, i.e. $K\left((i'-1),2\right)$,
($b_{30}$) take one step in the direction $_{i'2}d_{(i'-1)2}$, ($b_{31}$)
take $(j'-1)$ steps in the direction $_{i'3}d_{i',2}$ to reach the
maximum distance configuration. 

\textbf{(c)} Starting direction from $K(i',j')$ is $_{(i'+1)j'}d_{i'j'}$.
Maximum distance configuration is given in the steps shown below:

($c_{1}$) take $(2n-i')$ steps in the direction $_{(i'+1)j'}d_{i'j'}$
to reach last row, ($c_{2}$) take $(j'-1)$ steps in the direction
$_{(2n)(j'-1)}d_{(2n)(j')}$ to reach first column, ($c_{3}$) take
one step in the direction $_{(2n-1)1}d_{(2n)1}$, ($c_{4}$) take
$(2n-2)$ steps in the direction $_{(2n-2)1}d_{(2n-1)1}$ to reach
first row, ($c_{5}$) take one step in the direction $_{12}d_{11}$,
($c_{6}$) take $(2n-2)$ steps in the direction $_{22}d_{12}$, ($c_{7}$)
repeat the steps similar to the steps ($c_{4}$) to ($c_{6}$) until
the chicken is located in the cell $\left((2n-1),(j'-3)\right)$$ $,
i.e. $K\left((2n-1),(j'-3)\right)$, ($c_{8}$) take two steps in
the direction $_{(2n-1)(j'-2)}d_{(2n-1)(j'-3)}$, ($c_{9}$) take
one step in the direction $_{(2n-2)(j'-1)}d_{(2n-1)(j'-1)}$, ($c_{10}$)
take one step in the direction $_{(2n-2)(j'-2)}d_{(2n-2)(j'-1)}$, 

($c_{11}$) take one step in the direction $_{(2n-3)(j'-2)}d_{(2n-2)(j'-2)}$,
($c_{12}$) take one step in the direction

$_{(2n-3)(j'-1)}d_{(2n-3)(j'-2)}$, ($c_{13}$) repeat the steps similar
to the steps ($c_{10}$) to ($c_{12}$) to reach the cell $\left(1,(j'-1)\right)$,
i.e. $K\left(1,(j'-1)\right)$, ($c_{15}$) take $(2n-j'+1)$ steps
in the direction $_{1j'}d_{1(j'-1)}$ to reach last column, ($c_{15}$)
take $(2n-1)$ steps in the direction $_{2(2n)}d_{1(2n)}$ to reach
last row, ($c_{16}$) take one step in the direction $_{(2n)(2n-1)}d_{(2n)(2n)}$,
($c_{17}$) take $(2n-2)$ steps in the direction $_{(2n-1)(2n-1)}d_{(2n)(2n-1)}$,
($c_{18}$) take one step in the direction$_{2(2n-2)}d_{2(2n-1)}$,
($c_{19}$) take $(2n-2)$ steps in the direction $_{3(2n-2)}d_{2(2n-2)}$,
($c_{20}$) repeat the steps similar to the steps ($c_{16}$) to ($c_{19}$)
such that the chicken is located in the cell $\left((2n),(j'+3)\right)$,
i.e. $K\left((2n),(j'+3)\right)$, ($c_{21}$) take one step in the
direction $_{(2n)(j'+2)}d_{(2n)(j'+3)}$, ($c_{22}$) take one step
in the direction $_{(2n)(j'+1)}d_{(2n)(j'+2)}$, ($c_{23}$) take
one step in the direction $_{(2n-1)(j'+1)}d_{(2n)(j'+1)}$, ($c_{24}$)
take one step in the direction $_{(2n-1)(j'+2)}d_{(2n-1)(j'+1)}$,
($c_{25}$) take one step in the direction $_{(2n-2)(j'+2)}d_{(2n-1)(j'+2)}$,
($c_{26}$) repeat the steps similar to the steps ($c_{22}$) to ($c_{25}$)
such that the chicken in located in the cell $\left(2,(j'+2)\right)$,
i.e. $K\left(2,(j'+2)\right)$, ($c_{27}$) take two steps in the
direction $_{2(j')}d_{2(j'+2)}$, ($c_{28}$) take $(i'+3)$ steps
in the direction $_{3j'}d_{2j'}$ such that the chicken reaches maximum
distance under the hypotheses. 

\textbf{(d)} Starting direction from $K(i',j')$ is $_{i'(j'-1)}d_{i'j'}$.
Maximum distance configuration is given in the steps shown below:

($d_{1}$) take $(j'-1)$ steps in the direction $_{i'(j'-1)}d_{i'j'}$
to reach the first column, ($d_{2}$) take $(i'-1)$ steps in the
direction $_{(i'-1)1}d_{i'1}$ to reach the first row, ($d_{3}$)
take $(2n-1)$ steps in the direction of $_{12}d_{11}$ to reach the
last column, ($d_{4}$) take one step in the direction $_{2(2n)}d_{1(2n)}$,
($d_{5}$) take $(2n-2)$ steps in the direction of $_{2(2n-1)}d_{2(2n)}$,
($d_{6}$) take one step in the direction $_{32}d_{22}$, ($d_{7}$)
take $(2n-2)$ steps in the direction $_{33}d_{32}$ to reach the
last column, ($d_{8}$) repeat the steps similar to the steps ($d_{4}$)
to ($d_{7}$) such that the chicken is located at $\left((i'-1),2n\right)$,
i.e. $K\left((i'-1),2n\right)$, ($d_{9}$) take $(2n-i'+1)$ steps
in the direction $_{i'(2n)}d_{(i'-1)(2n)}$ to reach the last row,
($d_{10}$) take one step in the direction $_{(2n)(2n-1)}d_{(2n)(2n)}$,
($d_{11}$) take $(2n-2)$ steps in the direction $_{(2n)(2n-2)}d_{(2n)(2n-1)}$
to reach the first column, ($d_{12}$) take one step in the direction
$_{(2n-1)1}d_{(2n)1}$, ($d_{13}$) take $(2n-2)$ steps in the direction
$_{(2n-1)2}d_{(2n-1)1}$, ($d_{14}$) take one step in the direction
$_{(2n-2)(2n-1)}d_{(2n-1)(2n-1)}$, ($d_{15}$) repeat the steps similar
to the steps ($d_{11}$) to ($d_{14}$) such that the chicken is located
at $\left(i',(2n-1)\right)$, i.e. $K\left(i',(2n-1)\right)$, ($d_{16}$)
take $(2n-j'-2)$ steps in the direction $_{i'(2n-2)}d_{i'(2n-1)}$
to reach the maximum distance configuration at the cel $(i',j'+1)$. 

For all the other positions of the chicken at the beginning, we can
formulate configurations in each of the four directions to reach the
maximum distance. 

\textbf{(II) $S$ has dimension $(2n+1)\times(2n+1)$. }We can obtain
configuration for the longest walk in all four directions as explained
in $2n\times2n$ situation.

\textbf{(B). }Note that for a $2n\times2n$ dimensional area of cells,
there are $(2n+1)\times(2n+1)$ vertices, and if a chicken walks on
these vertices then by \textbf{(A) }the maximum distance walked is
$(n+1)^{2}-1$. 

When $K(i',j')$ is a corner cell then it will have two directional
options and when $K(i',j')$ is in boundary row or boundary column
(other than corner cell), then it will have three directional options,
and all these situations can be derived from the previous configurations.\end{proof}
\begin{example}
\label{example1}Here is an example $S$ has dimension $(2n+1)\times(2n+1)$
for the Theorem \ref{thm:1}. Suppose a walk is initiated by $K(1,1)$
in square of $S$ with $5\times5$. One of the longest walk is observed
when $K(1,1)$ reaches $K(5,5)$ by $ $the path, $\Gamma$, constructed
as below:

\begin{figure}
\begin{eqnarray*}
\Gamma\left((1,1)\rightarrow\left(5,5\right)\right) & = & \left[\begin{array}{c}
K(1,1)\\
\downarrow\\
K(1,2)\\
\downarrow\\
\vdots\\
K(1,5)\\
\downarrow\\
K(2,5)\\
\downarrow\\
\vdots\\
K(2,1)\\
\downarrow\\
K(3,1)\\
\downarrow\\
\vdots\\
K(3,5)\\
\downarrow\\
K(4,5)\\
\downarrow\\
\vdots\\
K(4,1)\\
\downarrow\\
K(5,1)\\
\downarrow\\
\vdots\\
K(5,5)
\end{array}\right]\\
\end{eqnarray*}

\caption{Path from $K(1,1)$ to $K(5,5)$ in example \ref{example1} }

\end{figure}

This path, $\Gamma$, covered all the cells and number of units travelled
by $K(1,1)$ under the straight walk and non-overlapping hypotheses
is $5^{2}-1.$ Suppose $S$ has dimension $(2n\times2n)$ for $n>1$,
then the longest path cannot be constructed in the above pattern between
$K(1,1)$ and $K(2n,2n).$ When $S$ has dimension $(2n\times2n)$
for $n>1$, $2k$, then the longest path observed, for example, is
a walk between $K(1,1)$ and $K(1,2)$ or $K(1,1)$and $K(2,1)$ which
takes the distance of $2k^{2}-1$ units. We will see this in Theorem
\ref{theorem 2}. $ $By induction type argument, we can prove if
$S$ has dimension $2n\times2n$ then maximum distance walked is $(2n)^{2}-1$
and if $S$ has dimension $(2n+1)\times(2n+1)$ then the maximum distance
walked is $(2n+1)^{2}-1$. \end{example}
\begin{thm}
\label{theorem 2}When $S$ has dimension $2n\times2n$ $(n>$1) then
there always exists at least one configuration for which the walk
between $K(i,j)$ and $K(i',j')$ is maximum, i.e. $(2n)^{2}-1$ units,
under the hypotheses of straight walk and non-overlapping walk and
when $S_{ij}(i,j)$ and $S_{i'j'}(i',j')$ have two common vertices
between them or $ $$S_{ij}(i,j)$ and $S_{i'j'}(i',j')$ are adjacent
cells \textbf{(}Here $S_{ij}(i,j)$ and \textbf{$S_{i'j'}(i',j'$)
}should not be the corner cells). If $S_{ij}(i,j)$ and $S_{i'j'}(i',j')$
are non adjacent cells then there is no configuration under the same
hypotheses for which the walk between $K(i,j)$ and $K(i',j')$ is
maximum. \end{thm}
\begin{proof}
Suppose there are an odd number of rows to the $i^{th}$ row and an
even number of columns to the left of $j^{th}$ column. We are interested
in demonstrating a configuration where $K(i,j)$ walks to $S_{i'j'}(i',j')$.
We follow below steps to reach $S_{i'j'}(i',j').$ 

(i) take $(i-1)$ steps in the direction $_{(i-1)j}d_{ij}$ to reach
the first row, (ii) take $(j-1)$ steps in the direction $_{1(j-1)}d_{1j}$
to reach the first column, (iii) take one step in the direction $_{21}d_{11}$,
(iv) take one step in the direction $_{22}d_{21}$, (v) take one step
in the direction $_{32}d_{22}$, (vi) take one step in the direction
$_{31}d_{32}$ to reach the first column, (vii) take one step in the
direction $_{41}d_{31}$, (viii) repeat the steps similar to the steps
(vi) to (vii) to reach the cell $S_{(2n)1}(2n,1)$, (ix) take two
steps in the direction $_{(2n)2}d_{(2n)1}$, (x) take $(2n-2)$ steps
in the direction $_{(2n-1)3}d_{(2n)3}$, (xi) take one step in the
direction $_{24}d_{23}$, (xii) take $(2n-2)$ steps in the direction
$_{34}d_{24}$ to reach the last row, (xiii) take one step in the
direction $_{(2n)5}d_{(2n)4}$, (xiv) repeat the steps similar to
the steps (x) to (xiii) to reach the cell $S_{(2n)j}(2n,j)$, (xv)
take $(2n-i-1)$ steps in the direction $_{(2n-1)j}d_{(2n)j}$, (xvi)
take one step in the direction $_{(i+1)(j+1)}d_{(i+1)j}$, (xvii)
take $(2n-i-1)$ steps in the direction $_{(i+2)(j+1)}d_{(i+1)(j+1)}$
to reach the last row, (xviii) take one step in the direction $d_{(2n)(j+1)}$,
(xix) take $(2n-2)$ steps in the direction $_{(2n-1)(j+1)}d_{(2n)(j+1)}$,
(xx) take one step in the direction $_{2(j+2)}d_{2(j+1)}$, (xxi)
take $(2n-2)$ steps in the direction $_{3(j+1)}d_{2(j+1)}$ to reach
last row, (xxii) repeat the steps similar to the steps (xviii) to
(xxi) such that $K(i,j)$ reaches the cell $S_{(2n)(2n-2)}\left(2n,(2n-2)\right)$,
(xxiii) take two steps in the direction $_{(2n)(2n-1)}d_{(2n)(2n-2)}$
to reach the cell $S_{(2n)(2n)}\left(2n,2n\right)$, (xiv) take one
step in the direction $_{(2n-1)(2n)}d_{(2n)(2n)}$, (xv) take one
step in the direction $_{(2n-1)(2n-1)}d_{(2n-1)(2n)}$, (xvi) take
one step in the direction $_{(2n-2)(2n-1)}d_{(2n-1)(2n-1)}$, (xvii)
take one step in the direction $_{(2n-2)(2n)}d_{(2n-2)(2n-1)}$, (xiv)
take one step in the direction $_{(2n-3)(2n)}d_{(2n-2)(2n)}$, (xv)
repeat the steps similar to the steps (xxi) to (xxiv) to reach the
cell $S_{1(2n)}\left(1,2n\right)$, (xxvi) take $(2n-j-1)$ steps
in the direction $_{1(2n-1)}d_{1(2n)}$ to reach the cell $S_{1(2n-j-1)}$,
(xxvii) take $i$ steps in the direction $_{2(2n-j-1)}d_{1(2n-j-1)}$
to reach the cell $S_{i(j+1)}\left(i,(j+1)\right)$ which is our desired
$S_{i'j'}(i',j')$. Since we covered all the cells in this configuration
the distance covered is $4n^{2}-1.$ 

\begin{figure}
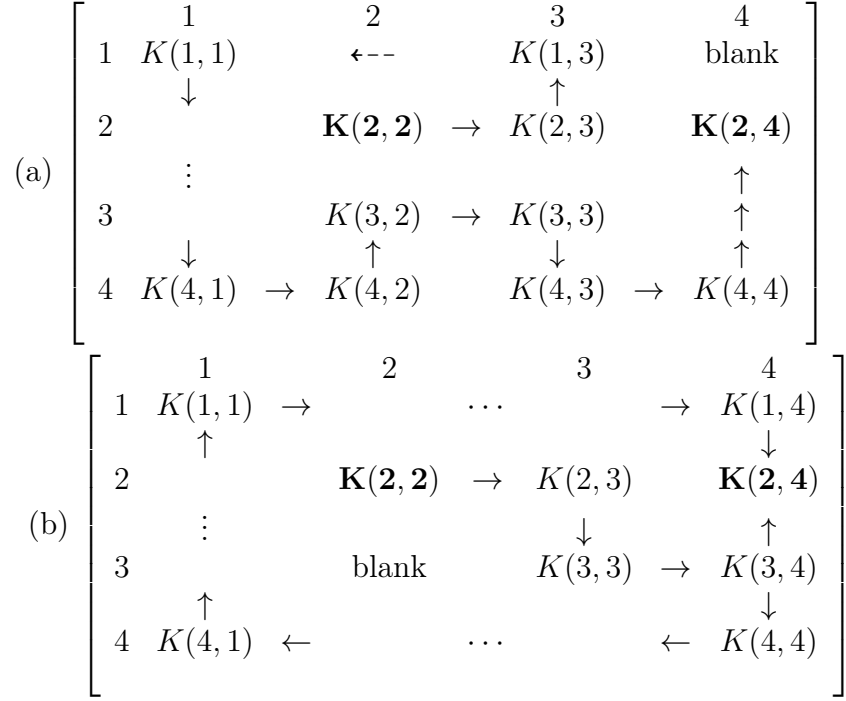

$ $(a) $\left[\begin{array}{ccccccccc}
 & 1 &  & 2 &  & 3 &  & 4\\
1 & K(1,1) &  & \dashleftarrow &  & K(1,3) &  & \mbox{blank}\\
 & \downarrow &  &  &  & \uparrow\\
2 &  &  & \mathbf{K(2,2)} & \rightarrow & K(2,3) &  & \mathbf{K(2,4)}\\
 & \vdots &  &  &  &  &  & \uparrow\\
3 &  &  & K(3,2) & \rightarrow & K(3,3) &  & \uparrow\\
 & \downarrow &  & \uparrow &  & \downarrow &  & \uparrow\\
4 & K(4,1) & \rightarrow & K(4,2) &  & K(4,3) & \rightarrow & K(4,4)\\
\\
\end{array}\right]$ $\quad$$ $

$ $

(b) $\left[\begin{array}{ccccccccc}
 & 1 &  & 2 &  & 3 &  & 4\\
1 & K(1,1) & \rightarrow &  & \cdots &  & \rightarrow & K(1,4)\\
 & \uparrow &  &  &  &  &  & \downarrow\\
2 &  &  & \mathbf{K(2,2)} & \rightarrow & K(2,3) &  & \mathbf{K(2,4)}\\
 & \vdots &  &  &  & \downarrow &  & \uparrow\\
3 &  &  & \mbox{blank} &  & K(3,3) & \rightarrow & K(3,4)\\
 & \uparrow &  &  &  &  &  & \downarrow\\
4 & K(4,1) & \leftarrow &  & \cdots &  & \leftarrow & K(4,4)\\
\\
\end{array}\right]$

\caption{Counter examples for second part of Theorem \ref{theorem 2}}
\label{Figure2-1}
\end{figure}

To prove second part, in contrary, let us assume that there exists
a configuration to obtain a maximum distance walked between $K(i,j)$
and $K(i',j')$ in any $S$ with $2n\times2n$ $(n\geq$1) dimension
when $K(i,j)$ and $K(i',j')$ are not adjacent. We bring one counter
example with configuration for two walks for which our assumption
fails to satisfy. Let us consider $K(i,j)=K(2,2)$ and $K(i',j')=K(2,4)$
and in $S$ with $4\times4$ dimension as shown in Figure \ref{Figure2-1}.
Both the walking paths configurations shown in Figure \ref{Figure2-1}(a)
and Figure \ref{Figure2-1}(b) have a distance covered 14 units less
than $\left(2.2\right)^{2}-1$ $units$. We can verify that other
walking paths from $K(2,2)$ to $K(2,4)$ would be less than $14$
$units$ or less. This is a contradiction to the hypothesis and that
proves the second part of the theorem.\end{proof}
\begin{example}
This is an example demonstration for the first part of Theorem \ref{theorem 2}.
Let us construct a configuration of walks between $K(i,j)=K(6,3)$
and $K(i',j')=K(6,4)$ when $S$ has dimension $10\times10$ i.e.
for $n=5$ (See Figure \ref{figure 2-2}). The trick to construct
such a walk depends on number of blank columns available before the
column in which $K(i,j)$ is located and number of blank columns available
after the column in which $K(i',j'$) s located. If the number of
blank columns are even then the configuration is given Figure \ref{figure 2-2}.
If the number of blank columns are odd on both the sides of $K(i,j)$
and $K(i',j')$, then for $K(5,4)$ and $K(5,5)$ adjacent squares
in $S$ with $8\times8$, we have given configuration in Figure \ref{Figure 2-3}.
In both of these examples, we saw that the distance walked was $\left(2.5\right)^{2}-1.$
Similar configuration structure can be used for higher dimension.
Instead the pair $K(5,4)$ and $K(5,5)$ in the Figure \ref{Figure 2-3},
suppose we are given, $K(5,4)$ and $K(4,4)$ to construct the configuration
for the longest walk. If we rotate Figure \ref{Figure 2-3} on its
right, the position of the cells $K(5,4)$ and $K(4,4)$ are similar
to the cells $K(5,4)$ and $K(5,5)$ before rotation. Hence the similar
configuration can be used after rotation and maximum distance walked
by $K(5,4)$ to reach $K(4,4)$ is also $\left(2.5\right)^{2}-1.$
The configuration to obtain maximum distance walked from $K(5,4)$
to $K(6,4)$ in Figure \ref{Figure 2-3} is similar to the one demonstrated
in the Figure \ref{figure 2-2}, because after rotation of $S$, the
number of blank rows on the left of the cell $(6,4)$ (which has become
$(4,3)$ after rotation) are even numbered. Similarly, the configuration
to obtain maximum distance walked from $K(5,4)$ to $K(5,3)$ in Figure
\ref{Figure 2-3} is similar to the one demonstrated in the Figure
\ref{figure 2-2}, because after rotation of $S$, the number of blank
rows on the left of the cell $(5,3)$ (which has become $(3,4)$ after
rotation) are even numbered. When $S$ has any $2n\times2n$ $(n\geq$1)
$ $dimension, we can configure a maximum distance walk in one of
the types discussed above. \end{example}
\begin{cor}
\label{cor:(max possible walks)}The total number of distinct pairs
of $K(i,j)$ and $K(i',j')$ in $S$ with dimension $2n\times2n$
$(n>$1) which are connected by maximum walks under the assumptions
of Theorem \ref{theorem 2} are $\left[\left(2n\right)\left\{ \left(2n\times2\right)-2\right\} \right]$.\end{cor}
\begin{proof}
For $n=2$, we have $4\times4$ cells and total number of pairs of
cells which satisfy criterion in Theorem \ref{theorem 2} are $24,$
which can be written as $\left[\left(2.2\right)\left\{ \left(2.2\times2\right)-2\right\} \right]$.
For $n=3$, we have $6\times6$ cells and total number of pairs of
cells satisfying Theorem \ref{theorem 2} are $\left[\left(2.2\right)\left\{ \left(2.3\times2\right)-2\right\} \right]$.
By induction we can prove the total number of pairs in $2n\times2n$
cells, connected by maximum walks are $\left[\left(2n\right)\left\{ \left(2n\times2\right)-2\right\} \right]$.
\end{proof}
\begin{landscape}

\begin{figure}
\begin{alignat*}{1}
\left[\begin{array}{ccccccccccccccccccccc}
 & 1 &  & 2 &  & 3 &  & 4 &  & 5 &  & 6 &  & 7 &  & 8 &  & 9 &  & 10\\
1 & K(1,1) &  & \leftarrow &  & K(1,3) &  & \downarrow &  & \dashleftarrow &  & \dashleftarrow &  & \dashleftarrow &  & \dashleftarrow &  & \dashleftarrow & \leftarrow & K(1,10)\\
 & \downarrow &  &  &  &  &  &  &  &  &  &  &  &  &  &  &  &  &  & \uparrow\\
2 & K(2,1) & \rightarrow & K(2,2) &  & \uparrow &  & \downarrow &  & K(2,5) & \rightarrow & K(2,6) &  & K(2,7) & \rightarrow & K(2,8) &  & K(2,9) & \rightarrow & K(2,10)\\
 &  &  & \downarrow &  &  &  &  &  & \uparrow &  & \downarrow &  & \uparrow &  & \downarrow &  & \uparrow\\
3 & K(3,1) & \leftarrow & K(3,2) &  & \uparrow &  & \downarrow &  &  &  &  &  &  &  &  &  & K(3,9) & \leftarrow & K(3,10)\\
 & \downarrow &  &  &  &  &  &  &  &  &  &  &  &  &  &  &  &  &  & \uparrow\\
4 & K(4,1) & \rightarrow & K(4,2) &  & \uparrow &  & \downarrow &  &  &  &  &  &  &  &  &  & K(4,9) & \rightarrow & K(4,10)\\
 &  &  & \downarrow &  &  &  &  &  &  &  &  &  &  &  &  &  & \uparrow\\
5 & K(5,1) & \leftarrow & K(5,2) &  & \uparrow &  & \downarrow &  &  &  &  &  &  &  &  &  & K(5,9) & \leftarrow & K(5,10)\\
 & \downarrow &  &  &  & \uparrow &  & \downarrow &  &  &  &  &  &  &  &  &  &  &  & \uparrow\\
6 & K(6,1) & \rightarrow & K(6,2) &  & \mathbf{K(6,3)} &  & \mathbf{K(6,4)} &  & \vdots &  & \vdots &  & \vdots &  & \vdots &  & K(6,9) & \rightarrow & K(6,10)\\
 &  &  & \downarrow &  &  &  &  &  &  &  &  &  &  &  &  &  & \uparrow\\
7 & K(7,1) & \leftarrow & K(7,2) &  & K(7,3) & \rightarrow & K(7,4) &  &  &  &  &  &  &  &  &  & K(7,9) & \leftarrow & K(7,10)\\
 & \downarrow &  &  &  &  &  & \vdots &  &  &  &  &  &  &  &  &  &  &  & \uparrow\\
8 & K(8,1) & \rightarrow & K(8,2) &  &  &  & \downarrow &  &  &  &  &  &  &  &  &  & K(8,9) & \rightarrow & K(8,10)\\
 &  &  & \downarrow &  &  &  &  &  &  &  &  &  &  &  &  &  & \uparrow\\
9 & K(9,1) & \leftarrow & K(9,2) &  & \uparrow &  &  &  &  &  &  &  &  &  &  &  & K(9,9) & \leftarrow & K(9,10)\\
 & \downarrow &  &  &  & \vdots &  &  &  & \uparrow &  & \downarrow &  & \uparrow &  & \downarrow &  &  &  & \uparrow\\
10 & K(10,1) & \rightarrow &  &  & K(10,3) &  & K(10,4) & \rightarrow & K(10,5) &  & K(10,6) & \rightarrow & K(10,7) &  & K(10,8) &  &  & \rightarrow & K(10,10)\\
\\
\end{array}\right]
\end{alignat*}

\caption{Configuration for straight and non-overlapping walk in a $10\times10$
when even number of blank columns are present before $K(5,4)$ and
$K(5,5).$}

\label{figure 2-2}
\end{figure}

\end{landscape}

\begin{landscape}$ $

\begin{figure}
$\left[\begin{array}{cccccccccccccccc}
 & 1 &  & 2 &  & 3 &  & 4 &  & 5 &  & 6 &  & 7 &  & 8\\
1 & K(1,1) &  & \cdots &  &  & \leftarrow & K(1,4) &  & K(1,5) &  & \dashleftarrow &  & \dashleftarrow &  & K(1,8)\\
 & \downarrow &  &  &  &  &  & \uparrow &  & \downarrow\\
2 &  &  & K(2,2) & \rightarrow & K(2,3) &  & \uparrow &  & \downarrow &  & K(2,6) & \rightarrow & K(2,7)\\
 &  &  & \uparrow &  & \downarrow &  & \uparrow &  & \downarrow &  & \uparrow &  & \downarrow\\
3 &  &  &  &  &  &  & \uparrow &  & \downarrow\\
 &  &  &  &  &  &  & \uparrow &  & \downarrow\\
4 & \vdots &  & \vdots &  & \vdots &  & \uparrow &  & \downarrow &  & \vdots &  & \vdots &  & \vdots\\
 &  &  &  &  &  &  & \uparrow &  & \downarrow\\
5 &  &  &  &  &  &  & \mathbf{K(5,4)} &  & \mathbf{K(5,5)}\\
\\
6 &  &  &  &  &  &  & K(6,4) & \rightarrow & K(5,5)\\
 &  &  &  &  &  &  & \uparrow &  & \downarrow\\
7 &  &  &  &  &  &  & \vdots &  & \vdots\\
 & \downarrow &  & \uparrow &  & \downarrow &  & \uparrow &  & \downarrow &  & \uparrow &  & \downarrow &  & \uparrow\\
8 & K(8,1) & \rightarrow & K(8,2) &  & K(8,3) & \rightarrow & K(8,4) &  & K(8,5) & \rightarrow & K(8,6) &  & K(8,7) & \rightarrow & K(8,8)
\end{array}\right]$

\caption{Configuration for straight and non-overlapping walk in a $8\times8$
when odd number of blank columns are present before $K(6,3)$ and
$K(6,4).$}
\label{Figure 2-3}
\end{figure}

\end{landscape}
\begin{thm}
\label{thm4}When $S$ has dimension $(2n+1)\times(2n+1)$ $(n\geq1)$
then there always exists at least one configuration for which the
walk between $K(i,j)$ and $K(i',j')$ is maximum, i.e. $(2n+1)^{2}-1$
units, under the hypotheses of straight walk and non-overlapping walk
and satisfying each of the following criteria: (i) when $S_{ij}(i,j)$
and $S_{i'j'}(i',j')$ are on a same main diagonal, (ii) when $S_{ij}(i,j)$
and $S_{i'j'}(i',j')$ are on same row or same column and separated
by at least one cell and these $S_{ij}(i,j)$ and $S_{i'j'}(i',j')$
are not located in the $2nd$ column or $2nd$ row and $2n^{th}$
column or $2n^{th}$ row.\end{thm}
\begin{proof}
Before generalizing, we will give some numerical demonstrations of
configuration of maximum walks. 

(i) Suppose $n=2$, we have an $S$ with $5\times5$. Let $K(i,j)=K(1,1)$
and $K(i',j')=K(4,4)$. Configuration for maximum walk from $K(1,1)$
to $K(4,4)$ is shown in Figure \ref{figureK11-K44}(a). This type
of configurations can be adopted for reaching $K(2n,2n)$ from $K(1,1)$
for higher dimensions $n>3$ as well. Similarly configurations for
maximum walks from $K(5,1)$ to $K(1,5)$ in Figure \ref{figureK11-K44}(b)
and from $K(7,1)$ to $K(6,2)$ in Figure \ref{figure K71-K62} can
be extended for other dimensions. There exists at least one walk which
covers the maximum distance under the straight and non-overlapping
walk to reach any two cells on the main diagonal.

(ii) Let us understand the configurations, when $K(3,1)$ walks to
the cell $S_{37}$ in a $7\times7$ dimension (See Figure \ref{K31-K37}),
when $K(3,1)$ walks to the cell $S_{35}$, i.e. same row separated
by three cells in the middle row and when $K(1,1)$ walks to the cell
$S_{15}$, same row separated by three cells in the top row of a $3\times3$
dimension. These configurations are given in Figure \ref{K11-K15,K31-K35},
Figure \ref{K11-K15,K31-K35}(a) and Figure \ref{K11-K15,K31-K35}(b).
If we need to construct a maximum walk between two cells in a column
then we rotate the square where we described configuration for rows
and then proceed in a similar pattern. The pattern of walk configured
above will be same for other dimensions. 
\end{proof}
\begin{figure}
a) $\left[\begin{array}{cccccccccc}
 & 1 &  & 2 &  & 3 &  & 4 &  & 5\\
1 & \mathbf{K(1,1)} & \rightarrow & \rightarrow & \rightarrow & \rightarrow & \rightarrow & \rightarrow & \rightarrow & K(1,5)\\
 &  &  &  &  &  &  &  &  & \downarrow\\
2 & K(2,1) & \leftarrow &  &  & \cdots &  &  & \leftarrow & K(2,5)\\
 & \downarrow\\
3 & K(3,1) & \rightarrow &  &  & \cdots &  &  & \rightarrow & K(3,5)\\
 &  &  &  &  &  &  &  &  & \downarrow\\
4 & K(4,1) & \rightarrow & \rightarrow & \rightarrow & \rightarrow & \rightarrow & \mathbf{K(4,4)} &  & \vdots\\
 &  &  &  &  &  &  &  &  & \downarrow\\
5 & K(5,1) & \leftarrow &  &  & \cdots &  &  & \leftarrow & K(5,5)
\end{array}\right]$

$ $

b) $\left[\begin{array}{cccccccccc}
 & 1 &  & 2 &  & 3 &  & 4 &  & 5\\
1 & K(1,1) & \rightarrow & K(1,2) &  & K(1,3) & \rightarrow & K(1,4) &  & \mathbf{K(1,5)}\\
 & \uparrow &  & \downarrow &  & \uparrow &  & \downarrow &  & \uparrow\\
2 & \uparrow &  &  &  &  &  &  &  & \uparrow\\
 & \uparrow &  &  &  &  &  &  &  & \uparrow\\
3 & \uparrow &  & \vdots &  & \vdots &  & \vdots &  & \uparrow\\
 & \uparrow &  &  &  &  &  &  &  & \uparrow\\
4 & \uparrow &  &  &  &  &  &  &  & \uparrow\\
 & \uparrow &  & \downarrow &  & \uparrow &  & \downarrow\vdots &  & \uparrow\\
5 & \mathbf{K(5,1)} &  & K(5,1) & \rightarrow & K(5,3) &  & K(5,4) & \rightarrow & K(5,5)
\end{array}\right]$$ $

$ $

\caption{Configurations for $n=2$ in Theorem \ref{thm4} }
\label{figureK11-K44}
\end{figure}

\begin{rem}
We can obtain configurations which are not satisfied by Theorem \ref{thm4},
for example, in a $5\times5$ area, if $K(2,1)$ has to walk to $S_{24}$,
there exists a configuration, but there doesn't for $K(2,1)$ to $S_{23}$.
Hence a general statement like the one in Theorem \ref{thm4} is not
applicable for the $2nd$ row. \end{rem}
\begin{thm}
\label{thm6}Given an $S$ with $(2n+1)\times(2n+1)$, all the pairs
$K(i,j)$ and $K(i',j')$, lying on $S_{ij}(i,j)$ and $S_{i'j'}(i',j')$
which are in same diagonals of cell size ($2n+1)$ for all $n\geq1$
can be connected by straight and non-overlapping walk with a maximum
distance. \end{thm}
\begin{proof}
For $n=1$ the result is true by the Theorem \ref{thm4}(i). For $n=2,$
the dimension of $S$ is $5\times5$ and concerned diagonals with
cell sizes are: $3$, $5$. We have two diagonals with cell size $3$.
Let us consider $K(i,j)=K(3,1)$ and $K(i',j')=K(2,2).$ The configuration
for a walk from $K(3,1)$ to $K(2,2)$ is given in Figure \ref{K31-K32}.
Similarly, other configurations for walks between cells in same diagonals
in $5\times5$ can be constructed. The results is true for diagonal
with cell size is $5$ using Theorem \ref{thm4}(i). For $n=3,$ the
dimension of $S$ is $5\times5$ and concerned diagonals with cell
sizes are: $3$, $5$, $7$. A configuration for walk between two
cells of a diagonal with cell size $3$ can be repeated as discussed
before in this proof. A configuration for a walk between two cells
of main diagonal with cell size $7$ can be constructed using Theorem
\ref{thm4}(i). We demonstrate a configuration for a walk between
two cells $S_{73}(7,3)$ to $S_{64}(6,4)$ in a diagonal with a size
of $5$ in Figure \ref{K73-K64}. The pattern of walks in these examples
can be extended for higher dimensions. For every higher dimension,
we will have similar configuration such that the condition is satisfied
for every diagonal of size $2n+1$ for $n\geq1$. 
\end{proof}
\begin{figure}
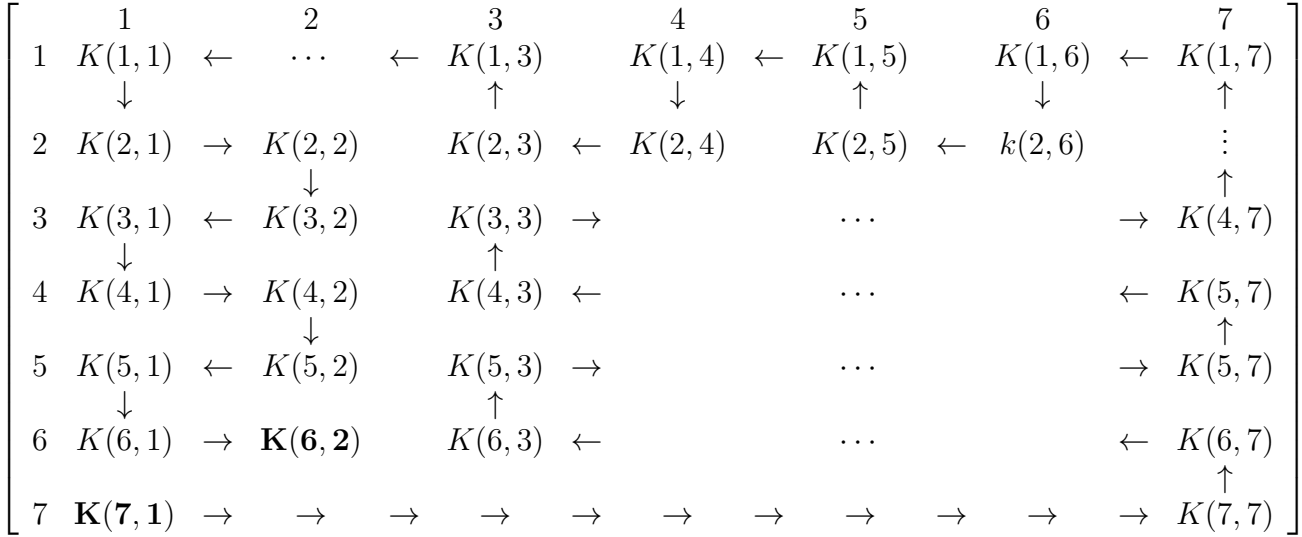

$\left[\begin{array}{cccccccccccccc}
 & 1 &  & 2 &  & 3 &  & 4 &  & 5 &  & 6 &  & 7\\
1 & K(1,1) & \leftarrow & \cdots & \leftarrow & K(1,3) &  & K(1,4) & \leftarrow & K(1,5) &  & K(1,6) & \leftarrow & K(1,7)\\
 & \downarrow &  &  &  & \uparrow &  & \downarrow &  & \uparrow &  & \downarrow &  & \uparrow\\
2 & K(2,1) & \rightarrow & K(2,2) &  & K(2,3) & \leftarrow & K(2,4) &  & K(2,5) & \leftarrow & k(2,6) &  & \vdots\\
 &  &  & \downarrow &  &  &  &  &  &  &  &  &  & \uparrow\\
3 & K(3,1) & \leftarrow & K(3,2) &  & K(3,3) & \rightarrow &  &  & \cdots &  &  & \rightarrow & K(4,7)\\
 & \downarrow &  &  &  & \uparrow\\
4 & K(4,1) & \rightarrow & K(4,2) &  & K(4,3) & \leftarrow &  &  & \cdots &  &  & \leftarrow & K(5,7)\\
 &  &  & \downarrow &  &  &  &  &  &  &  &  &  & \uparrow\\
5 & K(5,1) & \leftarrow & K(5,2) &  & K(5,3) & \rightarrow &  &  & \cdots &  &  & \rightarrow & K(5,7)\\
 & \downarrow &  &  &  & \uparrow\\
6 & K(6,1) & \rightarrow & \mathbf{K(6,2)} & \mathbf{} & K(6,3) & \leftarrow &  &  & \cdots &  &  & \leftarrow & K(6,7)\\
 &  &  &  &  &  &  &  &  &  &  &  &  & \uparrow\\
7 & \mathbf{K(7,1)} & \rightarrow & \rightarrow & \rightarrow & \rightarrow & \rightarrow & \rightarrow & \rightarrow & \rightarrow & \rightarrow & \rightarrow & \rightarrow & K(7,7)
\end{array}\right]$

\caption{Configurations for $n=3$ in Theorem \ref{thm4}(i)}

\label{figure K71-K62}
\end{figure}

\begin{landscape}

\begin{figure}
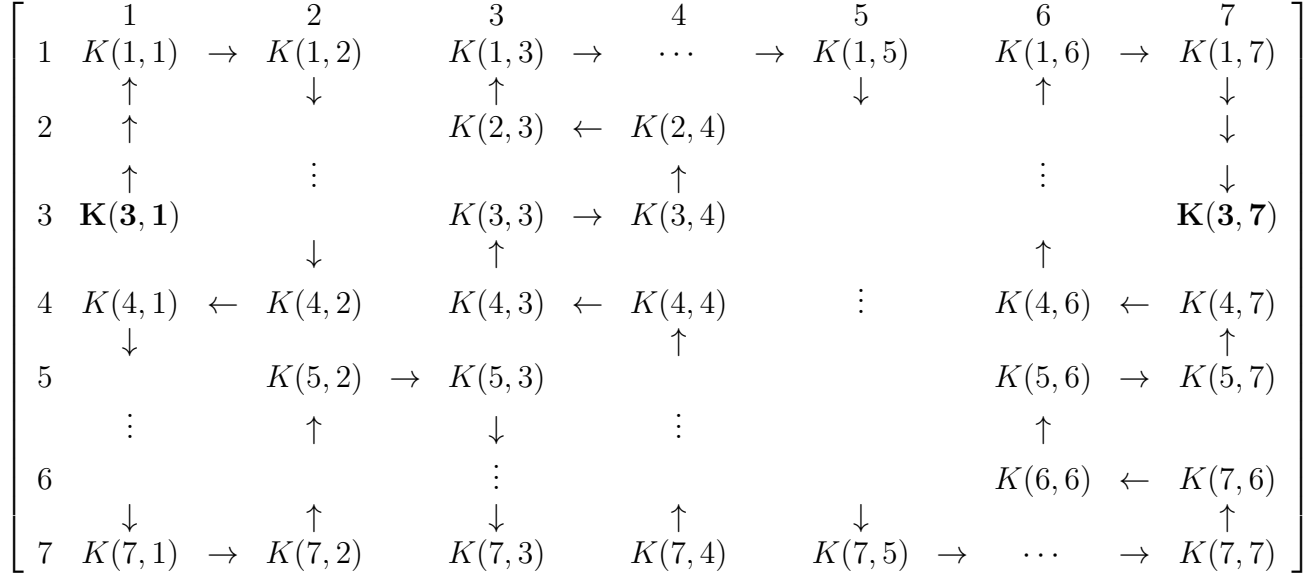

$\left[\begin{array}{cccccccccccccc}
 & 1 &  & 2 &  & 3 &  & 4 &  & 5 &  & 6 &  & 7\\
1 & K(1,1) & \rightarrow & K(1,2) &  & K(1,3) & \rightarrow & \cdots & \rightarrow & K(1,5) &  & K(1,6) & \rightarrow & K(1,7)\\
 & \uparrow &  & \downarrow &  & \uparrow &  &  &  & \downarrow &  & \uparrow &  & \downarrow\\
2 & \uparrow &  &  &  & K(2,3) & \leftarrow & K(2,4) &  &  &  &  &  & \downarrow\\
 & \uparrow &  & \vdots &  &  &  & \uparrow &  &  &  & \vdots &  & \downarrow\\
3 & \mathbf{K(3,1)} &  &  &  & K(3,3) & \rightarrow & K(3,4) &  &  &  &  &  & \mathbf{K(3,7)}\\
 &  &  & \downarrow &  & \uparrow &  &  &  &  &  & \uparrow\\
4 & K(4,1) & \leftarrow & K(4,2) &  & K(4,3) & \leftarrow & K(4,4) &  & \vdots &  & K(4,6) & \leftarrow & K(4,7)\\
 & \downarrow &  &  &  &  &  & \uparrow &  &  &  &  &  & \uparrow\\
5 &  &  & K(5,2) & \rightarrow & K(5,3) &  &  &  &  &  & K(5,6) & \rightarrow & K(5,7)\\
 & \vdots &  & \uparrow &  & \downarrow &  & \vdots &  &  &  & \uparrow\\
6 &  &  &  &  & \vdots &  &  &  &  &  & K(6,6) & \leftarrow & K(7,6)\\
 & \downarrow &  & \uparrow &  & \downarrow &  & \uparrow &  & \downarrow &  &  &  & \uparrow\\
7 & K(7,1) & \rightarrow & K(7,2) &  & K(7,3) &  & K(7,4) &  & K(7,5) & \rightarrow & \cdots & \rightarrow & K(7,7)
\end{array}\right]$

\caption{Configuration for $n=3$ in Theorem \ref{thm4}(ii)}
\label{K31-K37}
\end{figure}

\end{landscape}

\begin{figure}
a) $\left[\begin{array}{cccccccccc}
 & 1 &  & 2 &  & 3 &  & 4 &  & 5\\
1 & \mathbf{K(1,1)} &  & K(1,2) & \rightarrow & K(1,3) &  & K(1,4) & \rightarrow & \mathbf{K(1,5)}\\
 & \downarrow &  & \uparrow &  & \downarrow &  & \uparrow\\
2 &  &  &  &  &  &  & K(2,4) & \leftarrow & K(2,5)\\
 &  &  &  &  &  &  &  &  & \uparrow\\
3 & \vdots &  & \vdots &  & \vdots &  & K(3,4) & \rightarrow & K(3,5)\\
 &  &  &  &  &  &  & \uparrow\\
4 &  &  &  &  &  &  & K(4,4) & \leftarrow & K(4,5)\\
 & \downarrow &  & \uparrow &  & \downarrow &  &  &  & \uparrow\\
5 & K(5,1) & \rightarrow & K(5,2) &  & K(5,3) & \rightarrow & \cdots & \rightarrow & K(5,5)
\end{array}\right]$$ $

$ $

b) $\left[\begin{array}{cccccccccc}
 & 1 &  & 2 &  & 3 &  & 4 &  & 5\\
1 & K(1,1) & \rightarrow & K(1,2) &  & K(1,3) & \rightarrow & \cdots & \rightarrow & K(1,5)\\
 & \uparrow &  & \downarrow &  &  &  &  &  & \downarrow\\
2 & \vdots &  & \vdots &  & \vdots &  & K(2,4) & \leftarrow & K(2,5)\\
 & \uparrow &  &  &  &  &  & \downarrow\\
3 & \mathbf{K(3,1)} &  &  &  &  &  & K(3,4) & \rightarrow & \mathbf{K(3,5)}\\
 &  &  & \downarrow &  & \uparrow\\
4 & K(4,1) & \leftarrow & K(4,2) &  & K(4,3) & \leftarrow & \cdots & \leftarrow & K(4,5)\\
 &  &  &  &  &  &  &  &  & \uparrow\\
5 & K(5,1) & \rightarrow &  &  & \cdots &  &  & \rightarrow & K(5,5)
\end{array}\right]$

\caption{Configurations for $n=2$ in Theorem \ref{thm4}(ii)}
\label{K11-K15,K31-K35}
\end{figure}

\begin{figure}
$\left[\begin{array}{cccccccccc}
 & 1 &  & 2 &  & 3 &  & 4 &  & 5\\
1 & K(1,1) & \rightarrow &  &  & \cdots &  &  & \rightarrow & K(1,5)\\
 & \uparrow &  &  &  &  &  &  &  & \downarrow\\
2 & \uparrow &  & \mathbf{K(2,2)} &  & K(2,3) & \leftarrow & K(2,4)\\
 & \uparrow &  & \uparrow &  & \downarrow &  & \uparrow\\
3 & \mathbf{K(3,1)} &  & K(3,2) & \leftarrow & K(3,3) &  & \vdots &  & \vdots\\
 &  &  &  &  &  &  & \uparrow\\
4 & K(4,1) & \rightarrow &  &  & \cdots & \rightarrow & K(4,4)\\
 & \uparrow &  &  &  &  &  &  &  & \downarrow\\
5 & K(5,1) & \leftarrow &  &  & \cdots &  &  & \leftarrow & K(5,5)
\end{array}\right]$

\caption{Configuration for a walk between $K(3,1)$ to $K(2,2)$ in the proof
of Theorem \ref{thm6} }
\label{K31-K32}
\end{figure}

\begin{landscape}

\begin{figure}
$\left[\begin{array}{cccccccccccccc}
 & 1 &  & 2 &  & 3 &  & 4 &  & 5 &  & 6 &  & 7\\
1 & K(1,1) & \rightarrow & K(1,2) &  & K(1,3) & \rightarrow & K(1,4) &  & K(1,5) & \rightarrow & \cdots & \rightarrow & K(1,7)\\
 & \uparrow &  & \downarrow &  & \uparrow &  & \downarrow &  & \uparrow &  &  &  & \downarrow\\
2 &  &  &  &  &  &  &  &  &  &  & K(2,6) & \leftarrow & K(2,7)\\
 &  &  &  &  &  &  &  &  &  &  & \downarrow\\
3 & \vdots &  & \vdots &  & \vdots &  & \vdots &  & \vdots &  & K(3,6) & \rightarrow & K(3,7)\\
 &  &  & \downarrow &  & \uparrow &  &  &  &  &  &  &  & \downarrow\\
4 &  &  & K(4,2) & \rightarrow & K(4,3) &  &  &  &  &  & K(4,6) & \leftarrow & K(4,7)\\
 & \uparrow &  &  &  &  &  & \downarrow &  & \uparrow &  & \downarrow\\
5 & K(5,1) & \leftarrow & \cdots & \leftarrow & K(5,3) &  & K(5,4) & \rightarrow & K(5,5) &  & K(5,6) & \rightarrow & K(5,7)\\
 &  &  &  &  & \uparrow &  &  &  &  &  &  &  & \downarrow\\
6 & K(6,1) & \rightarrow & \cdots & \rightarrow & K(6,3) &  & \mathbf{K(6,4)} &  & K(6,5) & \leftarrow & K(6,6) &  & \vdots\\
 & \uparrow &  &  &  &  &  & \uparrow &  & \downarrow &  & \uparrow &  & \downarrow\\
7 & K(7,1) & \leftarrow & \leftarrow & \leftarrow & \mathbf{K(7,3)} & \mathbf{} & K(7,4) & \leftarrow & K(7,5) &  & K(7,6) & \leftarrow & K(7,7)
\end{array}\right]$

\caption{Configuration for a walk between $K(6,4)$ to $K(5,5)$ in the proof
of Theorem \ref{thm6} }
\label{K73-K64}
\end{figure}

\end{landscape}

\section{Rectifiable Paths}

Let $f_{1}:[K_{1},K_{2}]\rightarrow S\subset\mathbb{R}^{2}$ be a
path in $\mathbb{R}{}^{2}$, where $K_{1}$ is a starting point and
$K_{2}$ is an ending point of a maximum walk in some $ $$S$ with
a $2n\times2n$ area described in the previous section. In this section,
we study all the basic properties of paths generated by straight and
non-overlapping walks by $K(i,j)$. For the configuration explained
in the first part of the proof of the Theorem \ref{theorem 2}, we
divide into the following partition, $P_{1}$:

\begin{eqnarray*}
P_{1} & = & \left\{ p_{0},p_{1},\cdots,p_{(4n+2)},q_{1},q_{2},\cdots,q_{(2j-5)},\right.\\
 &  & \left.r_{1},r_{2},\cdots,r_{\left(\frac{2n-j-3}{2}\right)-1},s_{1}s_{2},\cdots,s_{n+8}\right\} 
\end{eqnarray*}

where $K_{1}=p_{0}$ and $K_{2}=s_{n+8}$ and the points $f_{1}(p_{0})$,
$f_{1}(p_{1})$, $\cdots,f_{1}(p_{(4n+2)})$, $f_{1}(q_{1})$, $\cdots,f_{1}(s_{(n+8)})$
are vertices (or the knots) of the polygon joining $(i,j)$ to $(i,j+1)$.
The set of vertices $\left\{ p_{0},p_{1},\cdots,p_{(4n+2)}\right\} $
join the cells from $(i,j)$ to $(2n,3)$, the set of vertices $\left\{ q_{1},q_{2},\cdots,q_{(2j-5)}\right\} $
join the cells $(2n,3)$ to $(2n,j)$, the set of vertices $\left\{ r_{1},r_{2},\cdots,r_{\left(\frac{2n-j-3}{2}\right)-1}\right\} $
join the cells $(2n-i-1,j)$ to $(2n,2n-1)$, the set of vertices
$\left\{ s_{1}s_{2},\cdots,s_{n+8}\right\} $ join the cells $(2n,2n)$
to $(i,j+1)$. The pairs of vertices $\left\{ p_{(4n+2)},q_{1}\right\} $,
$\left\{ q_{(2j-5)},r_{1}\right\} $, and $\left\{ r_{\left(\frac{2n-j-3}{2}\right)-1},s_{1}\right\} $
are also joined. The length of this polygon is 

\begin{eqnarray}
\Delta_{f_{1}}(P_{1}) & = & \Sigma_{h=1}^{4n+2}\left\Vert f_{1}(p_{h})-f_{1}(p_{h-1})\right\Vert +\left\Vert f_{1}(q_{1})-f(p_{(4n+2)})\right\Vert \nonumber \\
 &  & +\Sigma_{h=1}^{2j-5}\left\Vert f_{1}(q_{h})-f_{1}(q_{h-1})\right\Vert +\left\Vert f_{1}(r_{1})-f_{1}(q_{(2j-5)}\right\Vert \nonumber \\
 &  & +\Sigma_{h=1}^{\left(\frac{2n-j-3}{2}\right)}\left\Vert f_{1}(r_{h})-f_{1}(r_{h-1})\right\Vert +\left\Vert f_{1}(s_{1})-f_{1}(r_{\left(\frac{2n-j-3}{2}\right)})\right\Vert \nonumber \\
 &  & +\Sigma_{h=2}^{n+8}\left\Vert f_{1}(s_{h})-f_{1}(s_{h-1})\right\Vert \label{length}
\end{eqnarray}

The properties of the positioning of $K_{1}$ and $K_{2}$ i.e. the
number of columns and rows on the sides of $K_{1}$ and $K_{2}$ in
$S$ in the Theorem \ref{theorem 2} still holds here. 
\begin{lem}
\label{lemma-f1:-is-rectifiable.}$f_{1}:[K_{1},K_{2}]\rightarrow S\subset\mathbb{R}^{2}$
is rectifiable. \end{lem}
\begin{proof}
Since $(\ref{length})$ is bounded for all the combinations of vertices
joining the $K_{1}$ and $K{}_{2}$, the path $f_{1}$ is rectifiable.
(See \cite{Cesari1958} for rectifiable curves)\end{proof}
\begin{lem}
$f_{1}$ is of bounded variation (BV) on $[K_{1},K_{2}]$.\end{lem}
\begin{proof}
We have,

\begin{eqnarray*}
\left|\begin{array}{c}
\Sigma_{h=1}^{4n+2}\left|f_{1}(p_{h})-f_{1}(p_{h-1})\right|+\left|f_{1}(q_{1})-f(p_{(4n+2)})\right|\\
+\Sigma_{h=1}^{2j-5}\left|f_{1}(q_{h})-f_{1}(q_{h-1})\right|+\left|f_{1}(r1)-f_{1}(q_{(2j-5)}\right|\\
+\Sigma_{h=1}^{\left(\frac{2n-j-3}{2}\right)}\left|f_{1}(r_{h})-f_{1}(r_{h-1})\right|+\left|f_{1}(s_{1})-f_{1}(r_{\left(\frac{2n-j-3}{2}\right)})\right|\\
+\Sigma_{h=1}^{n+8}\left|f_{1}(s_{h})-f_{1}(s_{h-1})\right|
\end{array}\right| & < & 4n^{2}\\
\end{eqnarray*}

for all partitions of $[K_{1},K_{2}]$, so $f_{1}$ is of bounded
variation on $[K_{1},K_{2}]$. \end{proof}
\begin{thm}
Let $\mathbf{f}$ be a vector valued function defined as $\mathbf{f:}\left[K_{1},K_{2}\right]\rightarrow S\subset\mathbb{R}^{2}$
with components $\mathbf{f}=\left(f_{1},f_{2},\cdots,f_{k}\right)$,
then $\mathbf{f}$ is rectifiable.\end{thm}
\begin{proof}
We have seen that $f_{1}$ is rectifiable (see Lemma \ref{lemma-f1:-is-rectifiable.}).
Suppose $f_{2}:[K_{1},K_{2}]\rightarrow S\subset\mathbb{R}^{2}$.
The graph of $f_{2}$ drawn differently than $f_{1}$ in the sense
that, joining seven vertices beginning from $K_{1}$ we will arrive
at the cell $(2,3)$, and these seven cells are as follows:

$\left\{ (i,j)=p_{0},(1,j)=p_{1},(1,1)=p_{2},(2n,1)=p_{3},(2n,2)=p_{4},(2,2)=p_{5},(2,3)=p_{6}\right\} $.

Then, in the next two columns the pattern is similar to the one generated
in the steps (iii) to (ix) in the proof of Theorem \ref{theorem 2}
to reach the cell $(2n,5)$. By making such modifications in the graph,
the pattern of graph in the first two columns in $f_{1}$ is shifted
to columns $3$ and $4$, and the rest of the graph is remaining the
same. Now the partition, $P_{2}$ of $[K_{1},K_{2}]$ is 

\begin{eqnarray*}
P_{2} & = & \left\{ p_{0},p_{2},\cdots,p_{6},q_{1},\cdots,q_{(4n-5)},r_{1},r_{2},\cdots,r_{(2j-5)},\right.\\
 &  & \left.s_{1},s_{2},\cdots,s_{\left(\frac{2n-j-3}{2}\right)-1},t_{1}t_{2},\cdots,t_{n+8}\right\} 
\end{eqnarray*}

The length of this polygon is,

\begin{eqnarray*}
\Delta_{f_{2}}(P_{2}) & = & \Sigma_{h=1}^{6}\left\Vert f_{1}(p_{h})-f_{1}(p_{h-1})\right\Vert +\left\Vert f_{1}(q_{1})-f(p_{6})\right\Vert \\
 &  & +\Sigma_{h=1}^{4n-5}\left\Vert f_{1}(q_{h})-f_{1}(q_{h-1})\right\Vert +\left\Vert f_{1}(r_{1})-f(q_{(4n-5)})\right\Vert \\
 &  & +\Sigma_{h=1}^{2j-5}\left\Vert f_{1}(r_{h})-f_{1}(r_{h-1})\right\Vert +\left\Vert f_{1}(s_{1})-f_{1}(r_{(2j-5)}\right\Vert \\
 &  & +\Sigma_{h=1}^{\left(\frac{2n-j-3}{2}\right)}\left\Vert f_{1}(s_{h})-f_{1}(s_{h-1})\right\Vert +\left\Vert f_{1}(t_{1})-f_{1}(s_{\left(\frac{2n-j-3}{2}\right)})\right\Vert \\
 &  & +\Sigma_{h=1}^{n+8}\left\Vert f_{1}(t_{h})-f_{1}(t_{h-1})\right\Vert 
\end{eqnarray*}

Path, $f_{2}$ is rectifiable. We can partition $[K_{1},K_{2}]$ in
a different way, different to $P_{1}$ and $P_{2}$ and graph $f_{3}$
can be drawn differently by shifting the pattern of the graph of $f_{2}$
in columns (3) and (4) to the columns (5) and (6), and so on. We can
see all the components of $\mathbf{f}$ are of BV on $[K_{1},K_{2}]$.
Hence $\mathbf{f}$ is rectifiable. \end{proof}
\begin{thm}
\label{thm:vectorf continuous}Suppose $f_{1}:[K_{1},K_{2}]\rightarrow\mathbb{S\subset R}^{2}$,
$f_{2}:[K_{2},K_{3}]\rightarrow\mathbb{S\subset R}^{2}$, $\cdots,$
$f_{k}:[K_{k},K_{1}]\rightarrow S\subset\mathbb{R}^{2}$ are all possible
maximum walks in a $2n\times2n$ area ($K_{i}$ need not be in an
adjacent cell to $K_{i-1})$. Suppose these paths are overlapped either
partially or completely, but each path is continuous, then the vector
$\mathbf{f}=(f_{1},f_{2},\cdots,f_{k})$ is continuous. \end{thm}
\begin{proof}
$f_{1}$ is a path which describes a walk from $K_{1}$ to $K_{2}$
and $f_{2}$ is a path which describes a walk from $K_{2}$ to $K_{3}$
and so on, the piecewise combined paths are also continuous. Since
each path component is also continuous, $\mathbf{f}$ is also continuous. \end{proof}
\begin{rem}
By corollary \ref{cor:(max possible walks)}, we have $[2n\left\{ \left(2n\times2\right)-2\right\} ]$
paths until all possible maximum walks of Theorem \ref{thm:vectorf continuous}
are generated. We are interested in investigation of properties of
such walks. Whichever cell we initiate our walk from, all possible
maximum distances are covered in the process of generation of $\mathbf{f}$.
\end{rem}

\section{The Open Problem}

Instead of constructing rectifiable paths by allowing a movement through
adjacent rows and columns, here we allowed movements through adjacent
diagonals as well. Such a construction will lead to multiple possibilities
of maximum walks by starting at each cell, which we call trees of
paths. Trees are formed at each cell whose branches are rectifiable
paths. These trees, which are flexible and exhaustive, are helpful
in visualizing more realistic chicken walks on square grids. 

\textbf{Formulation of the problem: }Suppose an area $S$ consists
of $\left(2n\times2n\right)$ cells or $((2n+1)\times$ $(2n+1))$
cells. We start a straight and non-overlapping walk within S from
one of the cells, say, $\left(i,j\right)$ for $1<i<2n$ and $1<j<2n$.
We also allow diagonal moves to an unoccupied cell. Let us denote,
$K\left(i(t_{0}),j(t_{0})\right)$ for a walk which is initiated at
$t_{0}$ from the cell $\left(i,j\right)$, $K\left(i(t_{1}),j(t_{1})\right)$
is the position of this walk (or the position of the path generated
by this walk) at time $t_{1}$ and so on until a maximum possible
distance is achieved at $t_{m}$ (say). At $K\left(i(t_{0}),j(t_{0})\right)$
there are eight possible moves to the neighboring cells available
such that at time $t_{1}$ the path has reached one the following
positions: 

\begin{eqnarray}
\left\{ \begin{array}{ccc}
K\left((i-1)(t_{1}),j(t_{1})\right), & K\left((i-1)(t_{1}),(j+1)(t_{1})\right),\\
\\
K\left((i(t_{1}),(j+1)(t_{1})\right), & K\left((i+1)(t_{1}),(j+1)(t_{1})\right),\\
\\
K\left((i+1)(t_{1}),j(t_{1})\right), & K\left((i+1)(t_{1}),(j-1)(t_{1})\right),\\
\\
K\left(i(t_{1}),(j-1)(t_{1})\right), & K\left((i-1)(t_{1}),(j-1)(t_{1})\right)
\end{array}\right\} .\label{8 directions}\\
\nonumber 
\end{eqnarray}

Unless one or more of these positions in (\ref{8 directions}) are
located in the first or last row or in the first column or last column,
at each of these positions there are seven possible moves to reach
the neighboring cells at time $t_{2}$ (because one location is automatically
blocked by the non-overlapping hypothesis). Let us choose this to
be as $K\left((i+1)(t_{1}),(j-1)(t_{1})\right)$ and the seven walk
options are:

\begin{eqnarray}
\left\{ \begin{array}{ccc}
K\left(i(t_{2}),(j-1)(t_{2})\right), & \mbox{\textbf{blocked}}\left(i,j\right),\\
\\
K\left((i+1)(t_{2}),j(t_{2})\right), & K\left((i+2)(t_{2}),j(t_{2})\right),\\
\\
K\left((i+2)(t_{2}),(j-1)(t_{2})\right), & K\left((i+2)(t_{2}),(j-2)(t_{2})\right),\\
\\
K\left((i+1)(t_{2}),(j-2)(t_{2})\right) & K\left(i(t_{2}),(j-2)(t_{2})\right)
\end{array}\right\} .\label{7 directions}\\
\nonumber 
\end{eqnarray}

If walking path position at $t_{1}$ is located in the first or last
row or in the first column or last column (excepting in the four corner
cells), then there are four possible moves available to reach neighboring
cells at time $t_{2}$. At the next stage, i.e. at $t_{3}$, we have
at least six possible walking options for each of the seven previous
position in (\ref{7 directions}), unless at $t_{3}$ we arrive at
the first or last row or at the first column or last column. Similarly,
we can identify the number of possible options at each of the future
time points. By connecting cells from origin at $t_{0}$ through each
of the possible options at each of the time points, $t_{1}$, $t_{2}$,
$\cdots$, we will construct several rectifiable paths which have
maximum distances covered. Can we obtain a generalized formula for
the number of paths with maximum distances within $S$? 

For example, for $3\times3$, we will have $16$ maximum walks if
a path is initiated at $(3,3)$, $10$ maximum paths for each walk
if it is initiated at the first or last row or at the first column
or last column (excepting in the four corner cells), $6$ maximum
paths for each walk initiated at corner cells, which gives a total
number of paths with maximum possible distances in $3\times3$ area
are of $80$. Two examples of maximum paths in the $5\times5$ grid
are shown in Figure \ref{Fig1} and Figure \ref{fig:2}.

\begin{figure}
\includegraphics[scale=0.7]{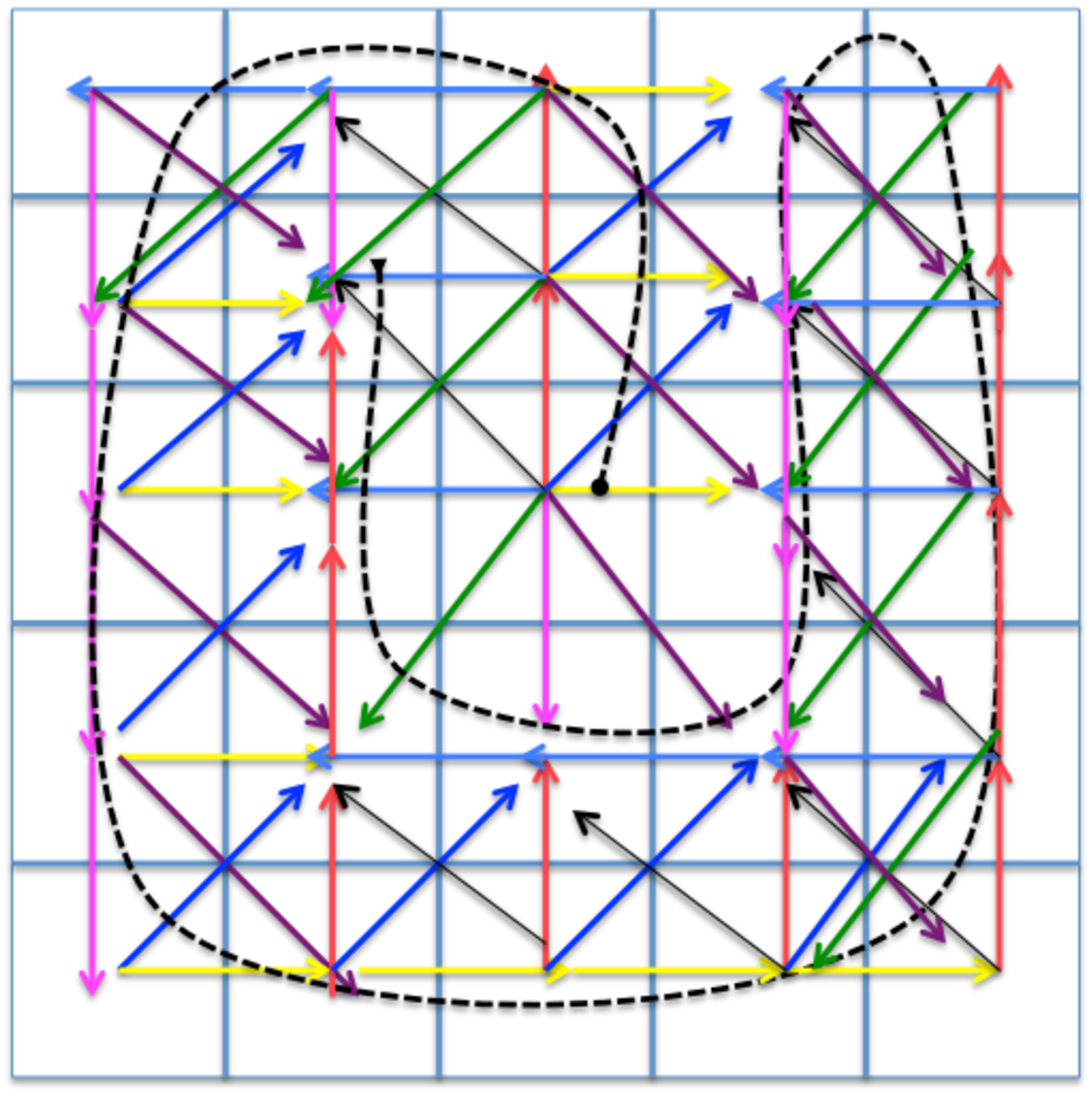}

\label{Fig1}\caption{Maximum path between cells $(3,3)$ and $(2,2)$ in the $5\times5$
grid. $ $Each colored arrow indicates the same direction but in a
different cell. All possible directions in each cell based on straight
and non-overlapping criteria are shown by arrow lines. Dotted line
is one of the longest possible paths \uline{without using a diagonal
movement} option. }
\end{figure}

\begin{figure}
\includegraphics[scale=0.7]{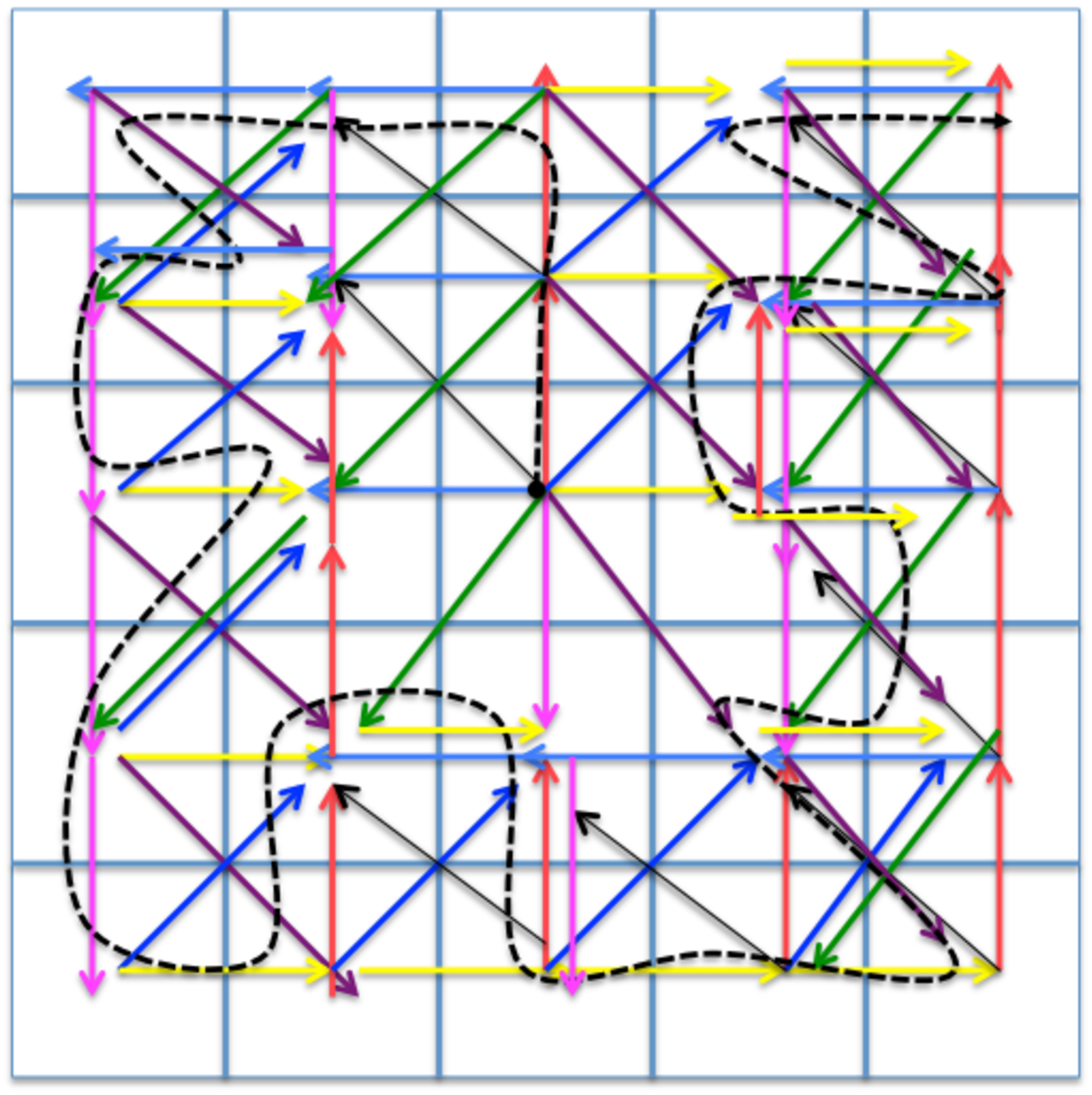}

\caption{\label{fig:2}Maximum path between cells $(3,3)$ and $(1,5)$ in
the $5\times5$ grid. $ $Each colored arrow indicates the same direction
but in a different cell. All possible directions in each cell based
on straight and non-overlapping criteria are shown by arrow lines.
Dotted line is one of the longest possible paths \uline{using a
diagonal movement} option. }

\end{figure}

\pagebreak

\section{Discussion }

Deconstructing movement into individual non-overlapping walks within
a grid-delineated area, which may then be strung together to model
the frequently repetitive, overlapping walks characteristic of the
domestic chicken, provides a framework to model faecal parasite dissemination.
Under the straight and non-overlapping set-up we are able to prove
conditions that prevent formation of maximum walks in a $2n\times2n$
grid (see Theorem \ref{theorem 2}), and in a \textbf{$(2n+1)\times(2n+1)$}
grid (see Theorem \ref{thm4}). In section 4, we have proved that
a vector of functions of bounded variations defined on a maximum possible
walk is rectifiable. By joining several rectifiable paths we arrive
at more meaningful chicken walks, which mimic several realistic situations
for understanding parasite transmissions.\textbf{ }Incorporating data
describing rate of defaecation (and thus parasite excretion) and previously
modelled transmission rates will then be key components in construction
of pathogen transmission models. Here we describe a mathematical model
defining host movement, in this case a chicken but it could be any
host, as the first tier of detail in the construction of a dynamic
model for transmission of a pathogen which is usually not airborne,
such as \emph{Eimeria.} Spatial placement of a chicken in its pen
or enclosure at any given time allows calculation of primary parasite
dissemination, providing a tool with which the frequency of opportunities
for neighbouring naive chickens to become infected may be predicted.
Extension of these calculations can be used to model pathogen transmission
through a flock. This approach can be adapted with relevant biological
parameters for any pathogen transmitted by direct or indirect physical
contact. 

We have provided a framework for understanding walks of chicken. By
joining several non-overlapping walks we get one complete walk of
a chicken per unit time. By joining several individual non-overlapping
walks, the resultant walk contains sub-walks which could be overlapped
and this is close to the reality of a flock of birds in a pen. Since
size of a cell within a grid is arbitrary, hence our analysis is flexible
to capture walks within very small pen sizes. Informed by this framework
each individual walk taken by a chicken may be portrayed across grids
through diagonal as well as non-diagonal dimensions. By joining multiple
paths we can define possible chicken behaviour over longer periods
of time. Marrying these behavioural measures with biological data,
including previously published rates of parasite transmission, we
hope to develop a method of understanding pathogen transmission dynamics
within the pens. One of our future aims of understanding chicken walks
is to predict the presence or absence of Eimeria in a chicken and
hence the proportion of infected chickens in a pen as an important
step towards transmission dynamics models for Eimeria. We wish to
study and build conjectures in general on association between the
longest paths of bird movement and disease dynamics. Other potential
applications for our chicken walk models include building age-structured
graphical models for chicken walks. One of our future aims of understanding
chicken walks is to predict the presence or absence of Eimeria in
a chicken and hence the proportion of infected chickens in a pen as
an important step towards transmission dynamics models for Eimeria.
We wish to study and build conjectures in general on association between
the longest paths of bird movement and disease dynamics. Other potential
applications for our chicken walk models include building age-structured
graphical models for chicken walks.

\section*{Acknowledgements}

Professor Lord Robert May (University of Oxford) encouraged new ideas
introduced in this work to study bird movements and gave useful comments,
Professor Tetali Prasad (Georgia Tech, Atlanta) suggested to draw
Figures in section 4, Professor Christopher Bishop (State University
of New York) suggested key references on Hamiltonian path problems.
Very useful corrections and comments by the two referees helped us
to rewrite several sentences and to add section 2 which improved overall
content of the paper. Professor N. Yathindra (Director, Institute
of Biotechnology and Applied Bioinformatics, Bangalore) and Professor
N.V. Joshi (Indian Institute of Science, Bangalore) introduced Arni
Rao to the Eimeria project initiated by the Royal Veterinary Collage,
London. Our sincere gratitude to all. This work was funded by BBSRC
(reference number BB/H009337/1). 

$ $

\begin{thebibliography}{10}
\bibitem[1]{Taylor2007}Taylor MA, Coop RL, and Wall RL Eds. \emph{Veterinary
Parasitology}, Blackwell Publishing Ltd. 2007.

\bibitem[2]{Chapman2013}Chapman HD, Barta JR, Blake D, Gruber A,
Jenkins M, Smith N, Suo X, Tomley FM. Review of coccidiosis research.
\emph{Advances in Parasitology} 2013;\textbf{83}:93-171.

\bibitem[3]{Dalloul2006-1}Dalloul R, Lillehoj, H. Poultry coccidiosis:
recent advancements in control measures and vaccine development. \emph{Expert
Review of Vaccines} 2006; 5: 143\textendash{}163. 

\bibitem[4]{Velkers}Velkers FC, Bouma A, Stegeman AJ, de Jong MCM.
Oocyst output and transmission rates during successive infections
with \emph{Eimeria} acervulina in experimental broiler flocks. \emph{Veterinary
Parasitology} 2012; 187:63-71.

\bibitem[5]{Shirley}Shirley MW, Smith AL, Tomley FM. The biology
of avian Eimeria with an emphasis on their control by vaccination.
\emph{Advances in Parasitology} 2005; \textbf{60}:285-330. 

\bibitem[6]{Williams}Williams RB, Johnson JD, Andrews SJ. Anticoccidial
vaccination of broiler chickens in various management programmes:
relationship between oocyst accumulation in litter and the development
of protective immunity. \emph{Veterinary Research Communications}
2000; \textbf{24}:309-325. 

\bibitem[7]{Willi}Williams RB. Epidemiological aspects of the use
of live anticoccidial vaccines for chickens. \emph{International Journal
for Parasitology} 1998; \textbf{28}:1089-1098. 

\bibitem[8]{Damer}Blake DP, Billington KJ, Copestake SL, Oakes RD,
Quail MA, Wan K-L, Shirley MW and Smith AL (2011) Genetic mapping
identifies novel highly protective antigens for an apicomplexan parasite.
PLoS Pathogens 7:e1001279

\bibitem[9]{ITAI}Itai A, Papadimitriou CH, Szwarcfiter JL. Hamilton
paths in grid graphs. \emph{SIAM Journal on Computing} 1982; \textbf{11}(4):
676\textendash{}686. 

\bibitem[10]{Zami}Zamfirescu C, Zamfirescu T. Hamiltonian properties
of grid graphs. \emph{SIAM Journal on Discrete Mathematics} 1992;
\textbf{5(}4):564\textendash{}570.

\bibitem[11]{Kesh}Keshavarz-Kohjerdi F, Bagheri A. Hamiltonian paths
in some classes of grid graphs. \emph{Journal of Applied Mathematics}
2012; Article ID 475087, 17 pages, doi:10.1155/2012/475087

\bibitem[12]{kwo}Kwong YH, Harris R, Rogers DG. A matrix method for
counting Hamiltonian cycles on grid graphs. \emph{European Journal
of Combinatorics} 1994; \textbf{15}(3):277\textendash{}283.

\bibitem[13]{Thom}Thompson GL. Hamiltonian Tours and Paths in Rectangular
Lattice Graphs. \emph{Mathematics Magazine 1977;} \textbf{50}(3):147\textendash{}150. 

\bibitem[14]{NASH}Nash-Williams CSJA. Infinite graphs \textendash{}
a survey. \emph{Journal Combinatorial Theory} 1967; \textbf{3}:286\textendash{}301. 

\bibitem[15]{NASH2}Nash\textendash{}Williams CSJA. Reconstruction
of infinite graphs. \emph{Discrete Mathematics} 1991;\textbf{95}:221\textendash{}229.

\bibitem[16]{BON}Bondy JA, Hemminger RL. Reconstructing infinite
graphs. \emph{Pacific Journal of Mathematics} 1974; \textbf{52}:331\textendash{}340.

\bibitem[17]{DS}Diestel R. Graph Theory, Springer, NY, 2000.

\bibitem[18]{HAR}Harel D. Hamiltonian paths in infinite graphs. \emph{Israel
Journal of Mathematics} 1991; \textbf{76}(3):317\textendash{}336.

\bibitem[19]{RODL}Rödl V, Ruci\'{n}ski A, Szemerédi E. Dirac-type
conditions for Hamiltonian paths and cycles in 3-uniform hypergraphs.
\emph{Advances in Mathematics} 2011;\textbf{3}:1225\textendash{}1299.

\bibitem[20]{Bellman1962}Bellman R. Dynamic programming treatment
of the travelling salesman problem\textquotedbl{}, \emph{Journal of
the ACM} 1962; \textbf{9}: 61\textendash{}63, doi:10.1145/321105.321111.

\bibitem[21]{Rubin1974}Rubin F. A Search Procedure for Hamilton Paths
and Circuits\textquotedbl{}, \emph{Journal of the ACM} 1974; \textbf{21}
(4): 576\textendash{}80, doi:10.1145/321850.321854.

\bibitem[22]{Ore1960}Ore O. Note on Hamilton circuits. \emph{American
Mathematical Monthly} 1960; \textbf{67:}55.

\bibitem[23]{Newman1958}Newman DJ. A problem in graph theory. American
Mathematics Monthly 1958\textbf{;65}:611.

\bibitem[24]{Kronk}Kronk HV. A note on k-path Hamiltonian graphs.
\emph{Journal of Combinatorial Theory} 1969;\textbf{7}:104\textendash{}106.

\bibitem[25]{Ajtai}Ajtai M, Komlós J, Szemerédi E. The longest path
in a random graph. \emph{Combinatorica }1981;\textbf{1},(1), 1\textendash{}12.

\bibitem[26]{Pittel}Pittel B. A random graph with a subcritical number
of edges. Transactions of American Mathematical Society 1988; \textbf{309}
(1): 51\textendash{}75.

\bibitem[27]{Krivlevich2013}Krivelevich M, Lubetzky E, Sudakov B.
Longest cycles in sparse random digraphs. \emph{Random Structures
\& Algorithms} 2013;\textbf{43}(1): 1\textendash{}15.

\bibitem[28]{Karger}Karger D, Motwani R, Ramkumar G. On approximating
the longest path in a graph, \emph{Algorithmica} 1997;\textbf{18}
(1):82\textendash{}98. 

\bibitem[29]{Feder}Feder T, Motwani R, Subi C, Approximating the
longest cycle problem in sparse graphs, \emph{SIAM Journal on Computing}
2002;\textbf{31} (5):1596\textendash{}1607.

\bibitem[30]{Zhang2011}Zhang WQ Liu YJ. Approximating the longest
paths in grid graphs\emph{. Theoretical Computer Science} 2011; \textbf{412}(39):
5340\textendash{}5350. 

\bibitem[31]{Fatemah}Keshavarz-Kohjerdi F, Bagheri A, Asghar AS.
A linear-time algorithm for the longest path problem in rectangular
grid graphs. \emph{Discrete Applied Mathematics} 2012; \textbf{160}(3):210-217. 

\bibitem[32]{Fatemah2013}Keshavarz-Kohjerdi F, Bagheri A. An efficient
parallel algorithm for the longest path problem in meshes. \emph{Journal
of Supercomputing} 2013;\textbf{65}:723-741.

\bibitem[33]{Fernandez}Fernandez de la Vega, W. Trees in sparse random
graphs.\emph{ Journal of Combinatorial Theory Series B} 1988;\textbf{45}(1):
77\textendash{}85. 

\bibitem[34]{Krive2010}Krivelevich M. Embedding spanning trees in
random graphs. \emph{SIAM Journal of Discrete Mathematics} 2010; \textbf{24}
(4): 1495\textendash{}1500.

\bibitem[35]{Chou2005}Chou, Arthur W.; Ko, Ker-I On the complexity
of finding paths in a two-dimensional domain. II. Piecewise straight-line
paths. \emph{Electronic Notes in Theoretical Computer Science} 2005
\textbf{120:}45-57.

\bibitem[36]{Henrici1986}Henrici P. \emph{Applied and computational
complex analysis}. Vol. 3. John Wiley \& Sons, Inc., New York, 1986.

\bibitem[37]{Cesari1958}Cesari L. Rectifiable Curves and the Weierstrass
Integral.\emph{ American Mathematical Monthly }1958; 67(7):485-500. 

\end{thebibliography}
\end{document}